\documentclass[11pt]{article}

\usepackage{fullpage}
\usepackage[utf8]{inputenc}
\usepackage[T1]{fontenc}  
\usepackage{todonotes}
\usepackage{amsmath}
\usepackage{amsfonts,amssymb}
\usepackage[margin=0.84in]{geometry}
\usepackage{url,hyperref}
\usepackage{amsthm}
\usepackage{xcolor}
\usepackage{float}
\usepackage{booktabs} 
\usepackage{graphicx}
\usepackage{caption}
\usepackage{subcaption}
\usepackage{ulem}
\usepackage{bm}

\usepackage[labelformat=simple]{subcaption}

\usepackage{multirow}

\newcommand{\BEAS}{\begin{eqnarray*}}
\newcommand{\EEAS}{\end{eqnarray*}}
\newcommand{\BEQ}{\begin{equation}}
\newcommand{\EEQ}{\end{equation}}
\newcommand{\BIT}{\begin{itemize}}
\newcommand{\EIT}{\end{itemize}}

\newcommand{\eg}{{\it e.g.}}
\newcommand{\ie}{{\it i.e.}}
\newcommand{\cf}{{\it cf. }}

\def\<#1,#2>{\langle #1,#2\rangle}

\newtheorem{theorem}{Theorem}
\newtheorem{remark}{Remark}
\newtheorem{assumption}{Assumption}
\newtheorem{lemma}[theorem]{Lemma}
\newtheorem{corollary}[theorem]{Corollary}
\newtheorem{proposition}[theorem]{Proposition}
\theoremstyle{definition}
\newtheorem{definition}{Definition}[section]

{\begin{quote}}{\end{quote}}

\makeatletter
\long\def\@makecaption#1#2{
   \vskip 9pt
   \begin{small}
   \setbox\@tempboxa\hbox{{\bf #1:} #2}
   \ifdim \wd\@tempboxa > 5.5in
        \begin{center}
        \begin{minipage}[t]{5.5in}
        \addtolength{\baselineskip}{-0.95pt}
        {\bf #1:} #2 \par
        \addtolength{\baselineskip}{0.95pt}
        \end{minipage}
        \end{center}
   \else
	\hbox to\hsize{\hfil\box\@tempboxa\hfil}
   \fi
   \end{small}\par
}
\makeatother

\newcounter{oursection}

\newcounter{lecture}

\newcommand{\norm}[1]{{\left\vert\kern-0.25ex\left\vert\kern-0.25ex\left\vert #1
    \right\vert\kern-0.25ex\right\vert\kern-0.25ex\right\vert}}

\newtheorem{example}{Example}

\title{MF-OMO: An Optimization Formulation of Mean-Field Games}
\date{}

\author{Xin Guo 
\thanks{University of California, Berkeley. 
\textbf{Email:}  \texttt{xinguo@berkeley.edu}} 
  \thanks{Amazon.com. \textbf{Email:} \texttt{xnguo@amazon.com}}
  \footnotemark[5]
    \and
Anran Hu 
\thanks{University of Oxford. 
\textbf{Email:} \texttt{Anran.Hu@maths.ox.ac.uk}}
\and
Junzi Zhang
\thanks{Citadel Securities. \textbf{Email:} \texttt{junzizmath@gmail.com}} 
}

\begin{document}

\maketitle

\begin{abstract}
  This paper proposes a new mathematical paradigm to analyze  discrete-time mean-field games. 
It is shown that finding  Nash equilibrium solutions for a general class of discrete-time mean-field games is equivalent to solving an optimization problem with bounded variables and simple convex constraints, called MF-OMO. This equivalence framework enables finding multiple (and possibly all) Nash equilibrium solutions of mean-field games by standard algorithms. For instance, projected gradient descent is shown to be capable of retrieving all possible  Nash equilibrium solutions when there are finitely many of them, with proper initializations. 
Moreover, analyzing mean-field games with linear rewards and mean-field independent dynamics is reduced to solving a finite number of linear programs, hence  solvable in finite time. This framework does not rely on  the contractive and the monotone assumptions and the uniqueness of the Nash equilibrium. 
\end{abstract}

\section{Introduction}\label{introduction}
Theory of  mean-field games, pioneered by the seminal work of  \cite{huang2006large} and  \cite{lasry2007mean}, provides an ingenious way of analyzing and finding the $\epsilon$-Nash equilibrium solution to the otherwise notoriously hard 
 $N$-player stochastic games. Ever since its inception, the literature of mean-field games has experienced an exponential growth  both in theory and in applications  (see the standard reference \cite{carmona2018probabilistic}).
 
 There are two key criteria for finding Nash equilibrium solutions of a mean-field game. The first is the optimality condition for the representative player, and the second is the consistency condition to ensure that the solution from the representative player is consistent with the mean-field flow of all players. Consequently, 
 mean-field games with continuous-time stochastic dynamics 
 have been studied 
 under three approaches. The first is the original fixed-point \cite{huang2006large} and PDE \cite{lasry2007mean} approach 
 that analyzes the joint system of  a backward Hamilton-Jacobi-Bellman equation for a control problem 
 and  a forward 
 Fokker-Planck equation  for the controlled dynamics. This naturally leads to the second probabilistic approach, which focuses on the associated forward-backward stochastic differential equations  \cite{buckdahn2009mean,carmona2013mean}. The last and the latest effort in mean-field game theory is to use master equation to analyze mean-field games with common noise and the propagation of chaos for optimal trajectories \cite{bensoussan2015master,cardaliaguet2019master,delarue2019master}. 
  All existing methodologies are based  on either the contractivity or  monotonicity conditions,  or the uniqueness of the Nash equilibrium solution, or some particular game structures; see  \cite{vasiliadis2019introduction,achdou2021mean} for a detailed summary. Accordingly, 
   finding Nash equilibrium solutions of mean-field games relies on  similar assumptions and structures; see for instance \cite{lauriere2021numerical}.

It is well known, however, that non-zero-sum stochastic games and mean-field games in general have multiple Nash equilibrium solutions and that uniqueness of the Nash equilibrium solution, contractivity, and monotonicity are more of an exception than a rule. Yet, systematic and efficient approaches for finding multiple or all Nash equilibrium solutions of general mean-field games remain elusive.

\paragraph{Our work and approach.}
In this work, we show that finding  Nash equilibrium solutions for a general class of discrete-time mean-field games is equivalent to solving an optimization problem with bounded variables and simple convex constraints, called MF-OMO (Mean-Field Occupation Measure Optimization) (Section \ref{MFG_opt_idea}).
This equivalence is established without assuming either the uniqueness of the Nash equilibrium solution or the contractivity  or monotonicity conditions. 

The equivalence is built through several steps. The first step, inspired by  the classical linear program formulation of the single-agent Markov decision process (MDP) problem \cite{puterman2014markov}, is to introduce occupation measure for the representative player, and to transform the problem of finding Nash equilibrium solutions to a constrained optimization problem (Theorem \ref{thm:nash-opt}) by the strong duality of the linear program to bridge the two criteria of Nash equilibrium solutions. 
The second step is to utilize a solution modification technique (Proposition \ref{interpretation_y_z}) to deal with potentially unbounded variables in the optimization problem and obtain the MF-OMO formulation (Theorem \ref{thm:mfvmo-full}). 
Further analysis of this equivalent optimization formulation enables establishing the connection between the suboptimality of MF-OMO and the exploitability of any policy for mean field games (Section \ref{sec:expl-vs-mfomo}). This is obtained by the perturbation analysis on MDPs (Lemma \ref{perturbation}) and by the strong duality of linear programs. 

These equivalence  and connections enable adopting families of optimization algorithms to find multiple (and possibly all) Nash equilibrium solutions of mean-field games. In particular,  detailed studies on the convergence of MF-OMO with the projected gradient descent (PGD) algorithm  to Nash equilibrium solutions of mean-field games are presented (Theorem \ref{local_conv_global}). For instance, PGD is shown to be capable of retrieving all possible Nash equilibrium solutions when there are finitely many of them, with proper initializations. Moreover, analyzing mean-field games with linear rewards and mean-field independent dynamics is shown to be reduced to solving a finite number of linear programs, hence solvable in finite time (Proposition \ref{finite-convergence-linear-mfg}). 

Finally, this MF-OMO framework can be  extended to other variants of mean-field games, as illustrated in the case of personalized mean-field games in Section \ref{refined_mfne}.

\paragraph{Related work.} Our approach is inspired by the classical linear program framework for MDPs, originally proposed in \cite{manne1960linear} for discrete-time MDPs with finite state-action spaces and then extended first to semi-MDPs \cite{osaki1968linear} and then to approximate dynamic programming \cite{schweitzer1985generalized,de2003linear}, as well as to continuous-time stochastic control  \cite{stockbridge1990time,bhatt1996occupation,kurtz1998existence} and singular control \cite{taksar1997infinite,kurtz2001stationary} with continuous state-action spaces, among  
others \cite{wolfe1962linear,derman1962sequential,denardo1970linear,hordijk1979linear}. 
In the similar spirit, our work builds an optimization framework for discrete-time mean-field games. 

There are several existing computational algorithms for finding Nash equilibrium solutions of discrete-time mean-field games. In particular, \cite{elie2020convergence,perolat2021scaling} assume strict monotonicity, while \cite{perrin2020fictitious} assumes weak monotonicity, all of which also assume the dynamics to be mean-field independent. In addition,  \cite{guo2019learning,fu2019actor,anahtarci2019fitted} focus on contractive mean-field games, 
\cite{angiuli2022unified} studies mean-field games with unique Nash equilibrium solutions, and 
\cite{cui2021approximately,anahtarci2022q}
analyze mean-field games that are sufficiently close to some contractive regularized mean-field games. {\color{black}Under these assumptions, global convergence is obtained for some of these algorithms. In contrast, we focus on presenting an equivalent optimization framework for mean-field games, and establishing local convergence without these assumptions.} 

The most relevant to our works are earlier works of \cite{bouveret2020mean,dumitrescu2021control,dumitrescu2022linear} on continuous-time mean-field games, in which they replace the best-response/optimality condition in Nash equilibrium solutions by the linear program optimality condition.
Their reformulation approach leads to new existence results for relaxed Nash equilibrium solutions and computational benefits. In contrast, our work formulates the  problem of finding Nash equilibrium solutions as a \textit{single} optimization problem which captures {\it both} the best-response condition and the Fokker-Planck/consistency condition. Our formulation shows that finding Nash equilibrium solutions of discrete-time mean-field games with finite state-action spaces is generally no harder than solving a smooth, time-decoupled optimization problem with bounded variables and trivial convex constraints. Such a connection allows for directly utilizing various  optimization tools, algorithms and solvers to find Nash equilibrium solutions of mean-field games. 
Moreover, our work provides, for the first time to our best knowledge, provably convergent algorithms for mean-field games with possibly multiple Nash equilibrium solutions and without contractivity or monotonicity assumptions.


\section{Problem setup}\label{prob_setup}
We consider an extended \cite{carmona2021probabilistic} or general \cite{guo2019learning} mean-field game in a finite-time horizon $T$ with a state space $\mathcal{S}$ and an  action space $\mathcal{A}$, where $|\mathcal{S}|=S<\infty$ and $|\mathcal{A}|=A<\infty$. In this game, a representative player starts from a state $s_0\sim\mu_0$, with $\mu_0$ the (common) initial state distribution of all players of an infinite population. At each time $t\in\mathcal{T}=\{0,1,\dots,T\}$ and 
when at state $s_t$, she chooses an action $a_t\in\mathcal{A}$ from some randomized/mixed policy $\pi_t:\mathcal{S}\rightarrow\Delta(\mathcal{A})$, where $\Delta(\mathcal{A})$ denotes the set of probability vectors on $\mathcal{A}$. Note that mixed policies are also known as relaxed controls. She will then move to a new state $s_{t+1}$ according to a transition probability $P_t(\cdot|s_t,a_t,L_t)$ and receive a reward $r_t(s_t,a_t,L_t)$, where $L_t\in\Delta(\mathcal{S}\times\mathcal{A})$ is the joint state-action distribution among  all players at time $t$,  referred to as the mean-field information hereafter, and $\Delta(\mathcal{S}\times\mathcal{A})$ is the set of probability vectors on $\mathcal{S}\times\mathcal{A}$. Denote $r_{\max}=\sup_{t\in\mathcal{T},s\in\mathcal{S},a\in\mathcal{A},L\in\Delta(\mathcal{S}\times\mathcal{A})}|r_t(s,a,L)|$, which is assumed throughout the paper as finite.

Given the mean-field flow $\{L_t\}_{t\in\mathcal{T}}$, the  objective of this representative player is to maximize her accumulated rewards, \textit{i.e.,} to solve the following MDP problem:
\begin{equation}\label{setup:MFG}
\begin{split}
\text{maximize}_{\{\pi_t\}_{t\in\mathcal{T}}}  \quad &\mathbb{E}\left[\sum_{t=0}^T r_t(s_t, a_t,L_t)\Big|s_0\sim \mu_0\right] \\
\text{subject to} \quad & {\color{black}s_{t+1}\sim P_t(\cdot|s_t,a_t,L_t)\, (t\in\mathcal{T}\backslash\{T\}),\quad a_t\sim \pi_t(\cdot|s_t)\, (t\in\mathcal{T}).}
\end{split}
\end{equation}
Throughout the paper, for a given mean-field flow $L=\{L_t\}_{t\in\mathcal{T}}$, we will denote   the mean-field induced MDP \eqref{setup:MFG} as $\mathcal{M}(L)$,   $V_t^\star(L)\in\mathbb{R}^S$ as its optimal total expected reward, \ie, value function, starting from time $t$,  with the $s$-th entry $[V_t^\star(L)]_s$ being the optimal expected reward starting from state $s$ at time $t$, and  $V_{\mu_0}^\star(L)=\sum_{s\in\mathcal{S}}\mu_0(s)[V_0^\star(L)]_s$ as   its optimal expected total reward starting from  $\mu_0$. 
Correspondingly,  we denote respectively  $V_t^{\pi}(L)\in\mathbb{R}^S$, $[V_t^{\pi}(L)]_s$, as well as  $V_{\mu_0}^{\pi}(L)=\sum_{s\in\mathcal{S}}\mu_0(s)[V_0^{\pi}(L)]_s$, 
under a given  policy $\pi=\{\pi_t\}_{t\in\mathcal{T}}$ for $\mathcal{M}(L)$. 

To analyze such a mean-field game, the most widely adopted solution concept is the Nash equilibrium. A policy sequence $\{\pi_t\}_{t\in\mathcal{T}}$ and a mean-field flow  $\{L_t\}_{t\in\mathcal{T}}$ constitute a Nash equilibrium solution of this finite-time horizon mean-field game, 
if the following conditions are satisfied.
\begin{definition}[Nash equilibrium solution]\label{mfne}
~
\begin{itemize}
    \item[1)] (Optimality/Best Response) Fixing $\{L_t\}_{t\in\mathcal{T}}$, $\{\pi_t\}_{t\in\mathcal{T}}$ solves the optimization problem \eqref{setup:MFG},
\textit{i.e.}, $\{\pi_t\}_{t\in\mathcal{T}}$ is optimal for the representative agent given the mean-field flow $\{L_t\}_{t\in\mathcal{T}}$; 
\item[2)] (Consistency) Fixing $\{\pi_t\}_{t\in\mathcal{T}}$, the consistency of the mean-field flow holds, namely
\begin{equation}\label{consistency_pop_flow}
L_t=\mathbb{P}_{s_t,a_t},\text{ where } s_{t+1}\sim P_t(\cdot|s_t,a_t,L_t),\, a_t\sim\pi_t(\cdot|s_t),\, s_0\sim \mu_0,\, t\in\mathcal{T}\backslash \{T\}. 
\end{equation}
Here $\mathbb{P}_{x}$ denotes the probability distribution of a random variable/vector $x$. 
\end{itemize} 
\end{definition}

Note that \eqref{setup:MFG} requires that the policy sequence $\{\pi_t\}_{t\in\mathcal{T}}$ is the best response to the flow $\{L_t\}_{t\in\mathcal{T}}$, while  \eqref{consistency_pop_flow} requires that the  flow $\{L_t\}_{t\in\mathcal{T}}$ is the corresponding mean-field flow induced when all players adopt the policy sequence $\{\pi_t\}_{t\in\mathcal{T}}$.\footnote{Condition 2) is also known  as the Fokker-Planck equation in the continuous-time mean-field game literature.} 
Also note that \eqref{consistency_pop_flow} can be written more explicitly as follows: 
\begin{equation}\label{consistency_def}
\begin{split}
    L_0(s,a)&=\mu_0(s)\pi_0(a|s), \\ L_{t+1}(s',a')&=\pi_{t+1}(a'|s')\sum_{s\in\mathcal{S}}\sum_{a\in\mathcal{A}}L_t(s,a)P_t(s'|s,a,L_t),\quad \forall t=0,\dots,T-1.
    \end{split}
\end{equation}

The following existence result for Nash equilibrium solutions holds as long as the transitions and rewards are continuous in $L_t$. The proof is based on the Kakutani fixed-point theorem; it is  almost identical to those in \cite{saldi2018markov, cui2021approximately}, except for replacing the state mean-field flow  with the state-action joint mean-field flow. 
\begin{proposition}\label{existence}
Suppose that $P_t(s'|s,a,L_t)$ and $r_t(s,a,L_t)$ are both continuous in $L_t$ for any $s,s'\in\mathcal{S}$, $a\in\mathcal{A}$ and $t\in\mathcal{T}$. Then a Nash equilibrium solution exists. 
\end{proposition}

Having established the existence of Nash equilibrium solution for the mean-field game (\eqref{setup:MFG} and \eqref{consistency_pop_flow}), the rest of the paper focuses on proposing a new analytical approach to find Nash equilibrium solution(s) for the mean-field game.  
This is motivated by  the well-known fact that  most mean-field games have  multiple Nash equilibrium solutions and are neither contractive nor (weakly) monotone, as the following simple example illustrates.

\paragraph{Example of multiple Nash equilibrium solutions.} 
Take $\mathcal{S}=\mathcal{A}=\{1,\dots,n\}$ and an arbitrary initial distribution $\mu_0$. At time $0$, for any $L_0\in\Delta(\mathcal{S}\times\mathcal{A})$ and any $i,j=1,\dots,n$, the rewards $r_0(i,j,L_0)=0$, and the transitions are deterministic 
such that 
$P_0(i'|i,j,L_0)=\mathbf{1}_{\{i'=j\}}$. 
When $t>0$, the rewards and the transition probabilities may depend on $L_t$. Specifically, if the population congregates at the same state $i$ and chooses to stay, \textit{i.e.}, $L_t(s,a)=\mathbf{1}_{\{s=i,a=i\}}$, then the rewards $r_t(i,j,L_t)=r^i\mathbf{1}_{\{i=j\}}$ for some $r^i>0$, and the transition probabilities will be deterministic $P_t(i'|i,j,L_t)=\mathbf{1}_{\{i'=j\}}$. If $L_t$ deviates from state $i$ and action $i$, the rewards and transition probabilities will be adjusted as follows: 

\begin{itemize}
    \item $r_t(i,j,L_t)=\mathbf{1}_{\{i=j\}}(r^i-r^i\|L_t-\mathbf{1}_{\{(s,a)|s=i,a=i\}}\|_2^2/2)$, for any $i,j=1,\dots,n$ and  $t=1,\dots,T$.\footnote{Here for a totally ordered finite set $\mathcal{X}$, we use $\mathbf{1}_{\mathcal{X}}$ to denote the binary vector in $\mathbb{R}^{|\mathcal{X}|}$ which has entry value $1$ for those indices $x\in\mathcal{X}$ and has entry value $0$ otherwise.}
    \item $P_t(i'|i,j,L_t)=\dfrac{\mathbf{1}_{\{i'=j\}}+C_{t}\|L_t-\mathbf{1}_{\{(s,a)|s=i,a=i\}}\|_2^2}{1+nC_{t}\|L_t-\mathbf{1}_{\{(s,a)|s=i,a=i\}}\|_2^2}$ for any $i',i,j=1,\dots,n$ and {\color{black}$t=1,\dots,T$}.
\end{itemize}
Here {\color{black}$C_t$ $(t=1,\dots,T)$} are some non-negative constants.

It is easy to verify that for any 
{\color{black}$j^\star\in\text{argmax}_{i=1,\dots,n}r^i$}, 
if we define \begin{itemize}
    \item $\pi_t(a|s)=\mathbf{1}_{\{a=j^\star\}}$ for any $s,a=1,\dots,n$ and $t=0,\dots,T$;
    \item $L_0(s,a)=\mu_0(s)\pi_0(a|s)$ and $L_t(s,a)=\mathbf{1}_{\{s=j^\star,a=j^\star\}}$, for any $s,a=1,\dots,n$ and $t=1,\dots,T$,
\end{itemize} 
then $\pi=\{\pi_t\}_{t\in\mathcal{T}}$ and $L=\{L_t\}_{t\in\mathcal{T}}$ constitute a Nash equilibrium solution. Thus there are several distinct Nash equilibrium solutions with both $\pi$ and $L$ being different, defined by the index {\color{black}$j^\star\in \text{argmax}_{i=1,\dots,n}r^i$} as $\pi^{(j^\star)}$ and $L^{(j^\star)}$. As contractivity and strict monotonicity imply  the uniqueness of the Nash equilibrium solution,  the mean-field game above is neither contractive nor strictly monotone. 
To see that it is not weakly monotone either, consider any two mean-field flows $L_t^{(1)}=L_t^{(j_1^\star)}$ and $L_t^{(2)}=L_t^{(j_2^\star)}$ with {\color{black}$j_1^\star\neq j_2^\star\in\text{argmax}_{i=1,\dots,n}r^i$}. Then
\begin{equation*}
\begin{split}
&\sum_{t=0}^T\sum_{s=1}^n\sum_{a=1}^n(L_t^{(1)}(s,a)-L_t^{(2)}(s,a))(r_t(s,a,L_t^{(1)})-r_t(s,a,L_t^{(2)}))\\
&={\color{black}T(r^{j_1^\star}+r^{j_2^\star})=2T\max\nolimits_{i=1,\dots,n}r^i}>0, 
\end{split}
\end{equation*}
violating the requirement of weak monotonicity. 

We now propose a new analytical framework that enables finding multiple Nash equilibrium solutions,  without either the   restrictive uniqueness assumption of the Nash equilibrium solution or   contractivity \cite{guo2019learning,cui2021approximately} and monotonicity \cite{perrin2020fictitious,perolat2021scaling} conditions.

\section{MF-OMO: Optimization formulation of mean-field games}\label{MFG_opt_idea} 
To overcome the restrictive assumptions of contractivity and monotonicity for analyzing mean-field games, we first establish a new optimization framework
to study mean-field games, via two steps. 
The first is an optimization reformulation of mean-field games with potentially unbounded variables and nonconvex constraints, which relies on the equivalence between an MDP problem and a linear program of \textit{occupation measures}. The second is to utilize a solution modification technique to obtain the final form of the optimization formulation with bounded variables and simple convex constraints. 
We call this framework  \textbf{M}ean-\textbf{F}ield \textbf{O}ccupation \textbf{M}easure \textbf{O}ptimization (MF-OMO).

\paragraph{Occupation measure.} 
To start, let us  introduce a new variable $d_t(s,a)$ for any $t\in\mathcal{T}$, $s\in \mathcal{S},a\in\mathcal{A}$, which represents the  occupation measure of the representative agent under some policy sequence $\pi=\{\pi_t\}_{t\in\mathcal{T}}$ in $\mathcal{M}(L)$, the mean-field induced MDP,
\textit{i.e.}, $d_t(s,a)=\mathbb{P}(s_t=s,a_t=a)$, with $s_0\sim\mu_0$, $s_{t+1}\sim P_t(\cdot|s_t,a_t,L_t)$, $a_t\sim \pi_t(\cdot|s_t)$ for $t=0,\dots,T-1$. 
Given the occupation measure $d=\{d_t\}_{t\in\mathcal{T}}$, define a mapping $\Pi$ that retrieves the policy from the occupation measure. This set-valued mapping $\Pi$ maps from a sequence $\{d_t\}_{t\in\mathcal{T}}\subseteq\Delta(\mathcal{S}\times\mathcal{A})$ to a set of policy sequences 
    $\{\pi_t\}_{t\in\mathcal{T}}$:  
    for any $\{d_t\}_{t\in\mathcal{T}}\subseteq\Delta(\mathcal{S}\times\mathcal{A})$, $\pi\in\Pi(d)$ if and only if
    \begin{equation}\label{eq:relation}
        \pi_t(a|s)=\frac{d_t(s,a)}{\sum_{a'\in\mathcal{A}}d_t(s,a')}
    \end{equation}
    when $\sum_{a'\in\mathcal{A}}d_t(s,a')>0$, and $\pi_t(\cdot|s)$ is an arbitrary probability vector in $\Delta(\mathcal{A})$ when  $\sum_{a'\in\mathcal{A}}d_t(s,a')=0$.  

\paragraph{Linear program formulation of MDPs.}
Next, consider any finite-horizon MDP $\mathcal{M}$ (\eg, the mean-field induced MDP $\mathcal{M}(L)$) with a finite state space $\mathcal{S}$, a finite action space $\mathcal{A}$, an initial state distribution $\mu_0$, a reward function $r_t(s,a)$, and dynamics $P_t(s'|s,a)$ for $t=0,\dots,T$, $s,s'\in\mathcal{S}$ and $a\in\mathcal{A}$. For brievity, we use $\mathbb{P}^{\pi}(\cdot)$  for the state and/or action distribution generated by 
$s_0\sim\mu_0$, $s_{t+1}\sim P_t(\cdot|s_t,a_t)$, $a_t\sim\pi_t(\cdot|s_t)$ for $t=0,\dots,T-1$. 

Then, our first component of MF-OMO relies on a linear program formulation of an MDP problem, as follows.  
 \begin{lemma}\label{prop:MDP_LP}
Suppose that $\pi=\{\pi_t\}_{t\in\mathcal{T}}$ is an $\epsilon$-suboptimal policy for the MDP $\mathcal{M}$. Define $\mu_t(s):=\mathbb{P}^{\pi}(s_t=s)$ and $d_t(s,a):=\mu_t(s)\pi_t(a|s)$ for $s\in\mathcal{S}$ and $a\in\mathcal{A}$. Then $\pi\in\Pi(d)$ and $d=\{d_t\}_{t\in\mathcal{T}}$ is a feasible $\epsilon$-suboptimal solution to the following linear program:
   \begin{align}
       \text{maximize}_{d}\quad &\sum_{t\in\mathcal{T}}\sum_{s\in\mathcal{S}}\sum_{a\in\mathcal{A}} d_t(s,a)r_t(s,a) \notag\\
       \text{subject to} \quad &\sum_{s\in\mathcal{S}}\sum_{a\in\mathcal{A}}d_t(s,a)P_t(s'|s,a)=\sum_{a\in\mathcal{A}}d_{t+1}(s',a), \quad\forall s'\in\mathcal{S},t\in\mathcal{T}\backslash\{T\},\notag\\
     &\sum_{a\in\mathcal{A}}d_0(s,a)=\mu_0(s), \quad \forall s\in\mathcal{S},\notag\\
       &d_t(s,a)\geq 0,\quad \forall s\in\mathcal{S},a\in\mathcal{A},t\in\mathcal{T}.\notag
    \end{align}
Conversely, suppose that $d=\{d_t\}_{t\in\mathcal{T}}$ is a feasible $\epsilon$-suboptimal solution to the above linear program. Then for any $\pi\in\Pi(d)$ defined by \eqref{eq:relation},  
 $\pi=\{\pi_t\}_{t\in\mathcal{T}}$ is an $\epsilon$-suboptimal policy for the MDP $\mathcal{M}$. In addition, the value of the above linear program is equal to the value of the MDP $\mathcal{M}$.
\end{lemma}
The proof  of Lemma \ref{prop:MDP_LP} is a suitable modification 
for its well-known counterpart in the discounted and average reward settings  \cite{de2003linear, huang1997lp, wang2016online}.

\paragraph{A preliminary version of MF-OMO.}
We now 
show that  $\pi=\{\pi_t\}_{t\in\mathcal{T}}$ is a Nash equilibrium solution if and only if there exist $d^\star=\{d_t^\star(s,a)\}_{t\in\mathcal{T}}$ and  $L=\{L_t(s,a)\}_{t\in\mathcal{T}}$, such that $\pi\in\Pi(d^\star)$ and the following two conditions  hold.

\begin{itemize}
    \item[\bf (A) ] Optimality of the representative agent: $d^\star=\{d_t^\star(s,a)\}_{t\in\mathcal{T}}$ solves the following LP problem 
    \begin{align}
       \text{maximize}_{d}\quad &\sum_{t\in\mathcal{T}}\sum_{s\in\mathcal{S}}\sum_{a\in\mathcal{A}} d_t(s,a)r_t(s,a,L_t) \\
      \text{subject to} \quad &\sum_{s\in\mathcal{S}}\sum_{a\in\mathcal{A}}d_t(s,a)P_t(s'|s,a,L_t)=\sum_{a\in\mathcal{A}}d_{t+1}(s',a),
            \label{cons_occ} \\ &\hspace{4.2cm}\forall s'\in\mathcal{S},t\in\mathcal{T}\backslash\{T\},\notag\\ 
      &\sum_{a\in\mathcal{A}}d_0(s,a)=\mu_0(s), \quad\forall s\in\mathcal{S},\notag\\
       &d_t(s,a)\geq 0,\quad \forall s\in\mathcal{S},a\in\mathcal{A},t\in\mathcal{T}.\notag
    \end{align}
    
        \item[\bf (B)] Consistency between the agent and the mean-field: 
$d_t^\star(s,a)=L_t(s,a)$ for any $s\in\mathcal{S},a\in\mathcal{A},t\in\mathcal{T}$.
\end{itemize}
Here condition (A) corresponds to the optimality condition 1) in Definition \ref{mfne} according to Lemma \ref{prop:MDP_LP}. Condition (B) indicates that the occupation measure of the single agent matches the mean-field flow of the population, so that the agent is indeed \textit{representative}. Condition (B), when combined with condition (A), in fact implies the consistency condition 2) in Definition \ref{mfne}. Formally, we have the following lemma.
\begin{lemma}\label{condition_A_B_prime}
If $(\pi,L)$ is a Nash equilibrium solution of the mean-field game {\rm(}\eqref{setup:MFG} and \eqref{consistency_pop_flow}{\rm)}, then there exists $d$ such that $\pi\in\Pi(d^\star)$ and $(d^\star,L)$ with $d^\star=L$ satisfying conditions (A) and (B). On the other hand, if $(d^\star,L)$ satisfies conditions (A) and (B), 
then for any $\pi\in\Pi(L)$,  
 $(\pi,L)$ is a Nash equilibrium solution of the mean-field game {\rm(}\eqref{setup:MFG} and \eqref{consistency_pop_flow}{\rm)}.
\end{lemma}

To get the preliminary version of MF-OMO, notice that the linear program in condition (A) can be rewritten as
\begin{equation}\label{cond_A_rewrite}
    \begin{array}{ll}
        \text{minimize}_d &c_L^\top d\\
        \text{subject to} &A_Ld=b,\quad d\geq0,
    \end{array}
    \end{equation}
    where $c_L\in\mathbb{R}^{SA(T+1)}$ and $b\in\mathbb{R}^{ST+S}$ are 
    \begin{equation}\label{kkt_param1}
    c_L=\left[
\begin{array}{c}
-r_0(\cdot,\cdot,L_0)\\
\vdots\\
-r_T(\cdot,\cdot,L_T)
\end{array}
\right],\quad
b=\left[
\begin{array}{c}
0\\
\vdots\\
0\\
\mu_0
\end{array}
\right],
    \end{equation}
    $r_t(\cdot,\cdot,L_t)\in\mathbb{R}^{SA}$ is a flattened vector (with column-major order),  
    and the matrix $A_L\in\mathbb{R}^{S(T+1)\times SA(T+1)}$ is
\begin{equation}\label{kkt_param2}
A_L=\left[
\begin{array}{ccccccc}
   W_0(L_0)  & -Z &  0 & 0 & \cdots &0 & 0\\
   0 & W_1(L_1)  & -Z &  0 & \cdots &0 & 0\\ 
   0 & 0 & W_2(L_2) & -Z & \cdots & 0 &0 \\
   \vdots & \vdots & \vdots & \vdots & \ddots & \vdots &\vdots \\
   0 & 0 & 0 & 0 & \cdots & W_{T-1}(L_{T-1}) & -Z\\
   Z & 0 & 0 & 0 & \cdots & 0 & 0
\end{array}\right].
\end{equation}
Here 
$W_t(L_t)\in \mathbb{R}^{S\times SA}$ is the matrix with the $l$-th row ($l=1,\dots,S$) being the  flattened vector  $[P_t(l|\cdot,\cdot,L_t)]\in\mathbb{R}^{SA}$ (with column-major order),  
and the matrix $Z$ is defined as 
\begin{equation}\label{eq:defn-Z}
 Z:=[\overbrace{I_{S},\dots,I_{S}]}^{A}\in \mathbb{R}^{S\times SA},  
\end{equation}
where $I_S$ is the identity matrix with dimension $S$. In addition, the variable $d$ is also flattened/vectorized in the same order as $c_L$. Accordingly, hereafter both $d^\star=\{d^\star_t\}_{t\in\mathcal{T}}$ and $L=\{L_t\}_{t\in\mathcal{T}}$ are viewed as flattened vectors (with column-major order) in $\mathbb{R}^{SA(T+1)}$ or a sequence of $(T+1)$ flattened vectors (with column-major order) in $\mathbb{R}^{SA}$, depending on the context, and we use $L_t(s,a)$ and $L_{s,a,t}$ (resp. $d_t(s,a)$ and $d_{s,a,t}$) alternatively.

By the strong duality of linear program \cite{luenberger1984linear}, $\{d_t^\star(s,a)\}$ is the optimal solution to the above linear program \eqref{cond_A_rewrite} if and only if the following KKT conditions hold for $d^\star$ and some $y$ and $z$:
\begin{equation}\label{eq:kkt}
\begin{split}
        &A_Ld^\star=b,\quad A_L^\top y+z=c_L,\\
        & z^\top d^\star=0,\quad d^\star\geq 0,\quad z\geq 0.
        \end{split}
    \end{equation}

Now, combining condition (A) (in the form of equation \eqref{eq:kkt}) with condition (B), we have the following preliminary version of MF-OMO.

\begin{theorem}  \label{thm:nash-opt}
Finding a Nash equilibrium solution of the mean-field game  {\rm(}\eqref{setup:MFG} and \eqref{consistency_pop_flow}{\rm)} is equivalent to solving the following 
constrained optimization problem:
    \begin{equation}\label{mf-vmo-constr}
    \begin{array}{ll}
    \text{\rm minimize}_{y,z,L} \quad &0\\
     \text{\rm subject to}  &A_LL=b,\quad A_L^\top y+z=c_L,\\
        & z^\top L=0,\quad L\geq 0,\quad z\geq 0,
        \end{array}
    \end{equation}
where $A_L,c_L,b$ are set as \eqref{kkt_param1} and \eqref{kkt_param2}.
Specifically, if $(\pi,L)$ is a Nash equilibrium solution of the mean-field game {\rm(}\eqref{setup:MFG} and \eqref{consistency_pop_flow}{\rm)}, then there exist some $y,z$ such that $(y,z,L)$ solves the constrained optimization problem \eqref{mf-vmo-constr}. On the other hand, if $(y,z,L)$ solves the constrained optimization problem \eqref{mf-vmo-constr}, then for any $\pi\in \Pi(L)$,  
$(\pi,L)$ is a Nash equilibrium solution of the mean-field game  {\rm(}\eqref{setup:MFG} and \eqref{consistency_pop_flow}{\rm)}.
\end{theorem}

{\color{black}  
In the above derivation, 
    we have used  $L$ to represent the population mean-field flows and $d$ to denote the occupation measure of the representative agent. The optimal occupation measure (\ie, the solution to the linear program) associated with the representative agent is denoted by $d^\star$. 
}

\paragraph{MF-OMO.} 
The above constrained optimization problem \eqref{mf-vmo-constr} in Theorem \ref{thm:nash-opt}
can be rewritten in a more computationally-efficient form  by adding bound constraints to the auxiliary variables $y$ and $z$ and moving potentially nonconvex constraints to the objective.  
Note that the consequent nonconvex objectives are in general much easier to handle than nonconvex constraints (\cf \eqref{mf-vmo-constr}). In addition, the boundedness of all variables is essential to the later analysis of the optimization framework in Sections \ref{sec:expl-vs-mfomo} and \ref{solve_mfvmo}. 
\begin{theorem}  \label{thm:mfvmo-full}
Finding a Nash equilibrium solution of the mean-field game  {\rm(}\eqref{setup:MFG} and \eqref{consistency_pop_flow}{\rm)} is equivalent to solving the following optimization problem with bounded variables and simple convex constraints:
\begin{equation}\label{mfvmo-full}
\begin{array}{ll}
\text{minimize}_{y,z,L} & \|A_LL-b\|_2^2+\|A_L^\top y+z-c_L\|_2^2+z^\top L\\
\text{subject to} & L\geq 0, \quad {\color{black}\mathbf{1}^\top L_t=1, \,t\in\mathcal{T},} \\
& \mathbf{1}^\top z\leq SA(T^2+T+2)r_{\max}, \quad z\geq 0, \tag{MF-OMO}\\
& \|y\|_2\leq \frac{S(T+1)(T+2)}{2}r_{\max},
\end{array}
\end{equation}
where $A_L,c_L,b$ are set as \eqref{kkt_param1} and \eqref{kkt_param2}, and ${\bf 1}$ is the all-one vector (with appropriate dimensions).
Specifically, if $(\pi,L)$ is a Nash equilibrium solution of the mean-field game {\rm(}\eqref{setup:MFG} and \eqref{consistency_pop_flow}{\rm)}, then there exist some $y,z$ such that $(y,z,L)$ solves \eqref{mfvmo-full}. On the other hand, if $(y,z,L)$ solves \eqref{mfvmo-full} with the value of the objective function being $0$, then for any $\pi\in \Pi(L)$,  
$(\pi,L)$ is a Nash equilibrium solution of the mean-field game {\rm(}\eqref{setup:MFG} and \eqref{consistency_pop_flow}{\rm)}.
\end{theorem}

Theorem \ref{thm:mfvmo-full} is immediate  from the following two results, which show how the bounds on the auxiliary variables $y$ and $z$ in \eqref{mfvmo-full} are obtained. Note that only $L$ is needed to construct the Nash equilibrium solution, hence there is no loss to add additional bounds on $y$ and $z$ according to Corollary \ref{y_z_bound}. 

The interpretation of $y$ and $z$ is motivated by the following remark.

\begin{proposition}[Solution modification/selection]\label{interpretation_y_z}
Suppose that $(y,z,L)$ is a solution to \eqref{mf-vmo-constr}. Define $\hat{y}$ as
\[
\hat{y}=[V_1^\star(L), V_2^\star(L), \dots,V_T^\star(L), -V_0^\star(L)]\in\mathbb{R}^{S(T+1)},
\]
where $V_t^\star(L)\in\mathbb{R}^S$ is the vector of value function starting from step $t$ for the induced MDP
$\mathcal{M}(L)$ (\cf Section \ref{prob_setup}).  
In addition, define $\hat{z}=[\hat{z}_0,\hat{z}_1,\dots,\hat{z}_T]\in\mathbb{R}^{SA(T+1)}$, where $\hat{z}_t$ is the Bellman residual/gap\footnote{This is also known as the (optimal) advantage function in the reinforcement learning literature.}, \ie, $\hat{z}_t=[\hat{z}_t^1,\dots,\hat{z}_t^A]$ ($t=0,\dots,T$) with 
\[
\hat{z}_t^a=V_t^\star(L)-r_t(\cdot,a,L_t)-P_t^a(L_t)V_{t+1}^\star(L)\in\mathbb{R}^S,\quad t=0,\dots,T-1,\,a\in\mathcal{A}
\]
and $\hat{z}_T^a=V_T^\star(L)-r_T(\cdot,a,L_T)\in\mathbb{R}^S$ ($a\in\mathcal{A}$), where $P_t^a(L_t)\in\mathbb{R}^{S\times S}$ is defined by $[P_t^a(L_t)]_{s,s'}=P_t(s'|s,a,L_t)$. 
Then $(\hat{y},\hat{z},L)$ is also a solution to \eqref{mf-vmo-constr}. 
\end{proposition}

The following corollary follows immediately by the definitions of $\hat{y}$ and $\hat{z}$ in Proposition \ref{interpretation_y_z}. 
\begin{corollary}\label{y_z_bound}
Suppose that $(y,z,L)$ is a solution to \eqref{mf-vmo-constr}. Then there exists a modified solution $(\hat{y}, \hat{z},L)$, such that the following bounds hold:
\[
\|\hat{y}\|_2\leq \|\hat{y}\|_1\leq \frac{S(T+1)(T+2)}{2}r_{\max},\quad \|\hat{z}\|_2\leq \|\hat{z}\|_1\leq SA(T^2+T+2)r_{\max}. 
\]
\end{corollary}
The key is to notice that $|[V_t^\star(L)]_s|\leq (T-t+1)r_{\max}$, 
$|\hat{z}_t^a(s)|\leq (T-t+1)r_{\max}+r_{\max}+(T-t)r_{\max}$ and $|\hat{z}_T^a(s)|\leq 2r_{\max}$, for any $s\in\mathcal{S},\,a\in\mathcal{A},\,t=0,\dots,T-1$. 

\begin{remark}\label{remark_mfvmo_form}
Note that the constraint  {\color{black}$\mathbf{1}^\top L_t=1$, $t\in\mathcal{T}$} is added to ensure the well-definedness of  $P_t$ and $r_t$  for non-optimal but feasible variables. As otherwise $P_t(s,a,\cdot)$ and $r_t(s,a,\cdot)$ are undefined outside of the simplex $\Delta(\mathcal{S}\times\mathcal{A})$. The choice of norms on $y$ and $z$ is to facilitate projection evaluation and reparametrization, where the first (projection evaluation) will be clear in Section \ref{solve_mfvmo} and the second (reparametrization) is explained in Appendix \ref{popular_tricks}. 
In particular, the choice of $\ell_1$-norm for $z$ is to obtain a simplex-alike specialized polyhedral constraint paired with $z\geq 0$,  and the $\ell_2$-norm for $y$ is to obtain closed-form projections (by normalization) and simple reparametrizations. Nevertheless, all our subsequent results in Sections \ref{sec:expl-vs-mfomo} and \ref{solve_mfvmo} hold (up to changes of constants) for other choices including $\ell_2$-norm for $z$ and $\ell_1$-norm for $y$, as long as either efficient projections or reparametrizations are feasible for $y$ and $z$.  
\end{remark}

\begin{remark}
Note that whenever a Nash equilibrium solution exists for the mean-field game {\rm(}\eqref{setup:MFG} and \eqref{consistency_pop_flow}{\rm)},  the optimal value of the objective function of \eqref{mfvmo-full} is obviously $0$, and vice versa. To see more explicitly   the relationship between \eqref{mfvmo-full} and  Nash equilibrium solution(s) of mean-field game {\rm(}\eqref{setup:MFG} and \eqref{consistency_pop_flow}{\rm)},  denote the feasible set of \eqref{mfvmo-full} as $\Theta$, \ie, 
\begin{equation}\label{Theta_feasible}
\begin{split}
\Theta:=\{(y,z,L)\,|\,&L\geq 0,{\color{black}\mathbf{1}^\top L_t=1,\,t\in\mathcal{T},}\, \\
&\|y\|_2\leq S(T+1)(T+2)r_{\max}/2,\\
&\mathbf{1}^\top z\leq SA(T^2+T+2)r_{\max},z\geq 0\}.
\end{split}
\end{equation}
Then \eqref{mfvmo-full} can be rewritten as minimizing the following objective function over $(y,z,L)\in\Theta$:
\begin{equation}\label{expansion-mfvmo}
\begin{array}{ll}
f^{\text{\rm MF-OMO}}(y,z,L)=& \sum\limits_{s\in\mathcal{S}}\left(\sum\limits_{a\in\mathcal{A}}L_0(s,a)-\mu_0(s)\right)^2\\
&+\sum\limits_{s'\in\mathcal{S}}\sum\limits_{t=0}^{T-1}\left(\sum\limits_{a\in\mathcal{A}}L_{t+1}(s',a)-\sum\limits_{s\in\mathcal{S},a\in\mathcal{A}}L_t(s,a)P_t(s'|s,a,L_t)\right)^2\\
&+\sum\limits_{s\in\mathcal{S},a\in\mathcal{A}}\left(y_{T-1}(s)-r_T(s,a,L_T)-z_T(s,a)\right)^2 \\
&+\sum\limits_{s\in\mathcal{S},a\in\mathcal{A}}\sum\limits_{t=0}^{T-2}\left(y_t(s)-r_{t+1}(s,a,L_{t+1})\right.\\
&\qquad\qquad\qquad\left.-\sum_{s'\in\mathcal{S}}P_{t+1}(s'|s,a,L_{t+1})y_{t+1}(s')-z_{t+1}(s,a)\right)^2\\
&+\sum\limits_{s\in\mathcal{S},a\in\mathcal{A}}\left(y_T(s)+r_0(s,a,L_0)+\sum\limits_{s'\in\mathcal{S}}P_0(s'|s,a,L_0)y_0(s')+z_0(s,a)\right)^2\\
&+\sum\limits_{s\in\mathcal{S},a\in\mathcal{A},t\in\mathcal{T}}z_t(s,a)L_t(s,a).
\end{array}
\end{equation}
Here the first two rows expand the term $\|A_LL-b\|_2^2$, corresponding to the consistency condition 2) in Definition \ref{mfne}. The next four rows expand the term $\|A_L^\top y+z-c_L\|_2^2$, corresponding to the Bellman residual by the interpretations of $y$ and $z$ in Proposition \ref{interpretation_y_z} and hence the optimality condition 1) in Definition \ref{mfne}. The last row expands the complementarity term $z^\top L$, connecting the two conditions into a single optimization problem.

This MF-OMO formulation is in sharp contrast to most existing approaches in the literature of discrete-time mean-field games such as \cite{guo2019learning, perrin2020fictitious, perolat2021scaling}, which alternate between a full policy optimization/evaluation given a fixed mean-field flow and a full mean-field flow forward propagation given a fixed policy sequence. Moreover, both the policy optimization/evaluation and the mean-field propagation lead to coupling over the full time horizon $\mathcal{T}$. As a result, the convergence of these algorithms requires assumptions of contractivity and monotonicity.  
In contrast, \eqref{expansion-mfvmo} shows that MF-OMO is fully decoupled over the time horizon, which better facilitates parallel and distributed optimization algorithms and enables more efficient sub-samplings for stochastic optimization algorithms, especially for large time-horizon problems. 
\end{remark}


\section{Connecting exploitability with MF-OMO suboptimality}\label{sec:expl-vs-mfomo}
Having established the equivalence between  minimizers of \eqref{mfvmo-full} and Nash equilibrium solutions of mean-field games, we now 
 connect the exploitablity of mean-field games with  suboptimal solutions to \eqref{mfvmo-full}. 
 
 The concept of exploitability in game theory is   used to characterize the difference between any policy and a Nash equilibrium solution. More precisely,  define a mapping $\Gamma$ that maps any policy sequence $\{\pi_t\}_{t\in\mathcal{T}}$ to $\{L_t\}_{t\in\mathcal{T}}$ when all players take such a policy sequence. Following the consistency condition \eqref{consistency_def}, such $\Gamma$ can be defined recursively, starting with the initialization 
\begin{equation}\label{consistency_explicit_1}
\Gamma(\pi)_0(s,a):=\mu_0(s)\pi_0(a|s), 
\end{equation} 
such that 
\begin{equation}\label{consistency_explicit_2}
\Gamma(\pi)_{t+1}(s,a):=\pi_{t+1}(a|s)\sum_{s'\in\mathcal{S}}\sum_{a'\in\mathcal{A}}\Gamma(\pi)_t(s',a')P_t(s|s',a',\Gamma(\pi)_t),\quad \forall t=0,\dots,T-1. 
\end{equation} 
For any given policy sequence $\pi=\{\pi_t\}_{t\in\mathcal{T}}$, let $L=\{L_t\}_{t\in\mathcal{T}}=\Gamma(\pi)$ be the induced mean-field flow. 
Then exploitability characterizes the sub-optimality of the policy $\pi$ under $L$ as follows, 
\begin{equation}\label{def:expl}
\text{Expl}(\pi):=V_{\mu_0}^\star(\Gamma(\pi))-V_{\mu_0}^{\pi}(\Gamma(\pi))=\max_{\pi'}V_{\mu_0}^{\pi'}(\Gamma(\pi))-V_{\mu_0}^{\pi}(\Gamma(\pi)).
\end{equation}
In particular, $(\pi,L)$ is a Nash equilibrium solution if and only if $L=\Gamma(\pi)$ and $\text{Expl}(\pi)=0$; 
and a policy $\pi$ is an $\epsilon$-Nash equilibrium solution if $\text{Expl}(\pi)\leq \epsilon$.  

However, computing and optimizing exploitability directly is not easy. First, $\text{Expl}(\pi)$ is generally nonconvex and nonsmooth, even if $r_t(s,a,\cdot)$ and $P_t(s,a,\cdot)$ are differentiable in $L_t$. 
Secondly, a single evaluation of $\text{Expl}(\pi)$ requires a full policy optimization when calculating $\max_{\pi'}V_{\mu_0}^{\pi'}(\Gamma(\pi))$ and a full policy evaluation when calculating $V_{\mu_0}^{\pi}(\Gamma(\pi))$.
In this section, we will show that in order to find an $\epsilon$-Nash equilibrium solution, it is sufficient to solve \eqref{mfvmo-full} to $O(\epsilon^2)$ 
sub-optimality.

We will need additional notation here before presenting the precise statement.
Denote $P_t(\cdot|\cdot,\cdot,L_t)\in\mathbb{R}^{S\times S\times A}$ as a tensor of dimensions $S\times S\times A$. 
For a sequence of $T+1$  tensors $p=\{p_t(\cdot|\cdot,\cdot)\}_{t\in\mathcal{T}}\in\mathbb{R}^{S\times S\times A}$ with dimensions $S\times S\times A$ (\eg, $p_t(s'|s,a)$ can be $P_t(s'|s,a,L_t)$) and a sequence of $T+1$ vectors  $x=\{x_t(\cdot,\cdot)\}_{t\in\mathcal{T}}\in\mathbb{R}^{SA}$ with dimensions $S\times A$ (\eg, $x_t(s,a)$ can be $r_t(s,a,L_t)$, $d_t(s,a)$ or $L_t(s,a)$, etc.), define  
\begin{equation*}
\begin{split}
&\|p\|_{\infty,1}=\max_{s\in\mathcal{S},a\in\mathcal{A},t=0,\dots,T-1}\sum_{s'\in\mathcal{S}}|p_t(s'|s,a)|,\quad \|x\|_{1,\infty}=\sum_{t\in\mathcal{T}}\max_{s\in\mathcal{S},a\in\mathcal{A}}|x_t(s,a)|.
\end{split}
\end{equation*}
Similarly, for any tensor $p_t(\cdot|\cdot,\cdot)\in\mathbb{R}^{S\times S\times A}$ with dimensions $S\times S\times A$, define 
\begin{equation*}
\begin{split}
&\|p_t\|_{\infty,1}=\max_{s\in\mathcal{S},a\in\mathcal{A}}\sum_{s'\in\mathcal{S}}|p_t(s'|s,a)|. 
\end{split}
\end{equation*}
Denote also $\|\cdot\|_1$, $\|\cdot\|_2$ and $\|\cdot\|_\infty$ respectively for the standard $\ell_1$, $\ell_2$ and $\ell_\infty$ vector and matrix norms. 
Then we will see that a near-optimal solution to \eqref{mfvmo-full} will be close to a Nash equilibrium solution, assuming  the dynamics and rewards are Lipschitz continuous in $L_t$. 
\begin{theorem}\label{VMO-vs-MFNE}
Suppose that for any $t\in\mathcal{T}$, $P_t(\cdot|\cdot,\cdot,L_t)$ and  $r_t(\cdot,\cdot,L_t)$ are $C_P(>0)$ and $C_r(>0)$ Lipschitz continuous in $L_t$, respectively, i.e., 
\begin{equation*}
 \|P_t^{L^1}-P_t^{L^2}\|_{\infty,1}\leq C_P\|L_t^1-L_t^2\|_1,\quad \|r_t^{L^1}-r_t^{L^2}\|_{\infty}\leq C_r\|L_t^1-L_t^2\|_1,
\end{equation*}
where $P_t^L:=P_t(\cdot|\cdot,\cdot,L_t)$ and $r_t^L:=r_t(\cdot,\cdot,L_t)$, and $L^i:=\{L_t^i\}_{t\in\mathcal{T}}$ ($i=1,2$) are two arbitrary mean-field flows. 
Let $y,z,L$ be a feasible solution to \eqref{mfvmo-full}, with the value of its objective function being $f^{\text{\rm MF-OMO}}(y,z,L)\leq \epsilon^2$ for some $\epsilon\geq 0$. Then for any $\pi\in\Pi(L)$, we have 
\begin{equation*}
\text{\rm Expl}(\pi)\leq f(S,A,T,C_P,C_r,r_{\max}) \epsilon+\epsilon^2,
\end{equation*}
where 
\begin{equation*}
\begin{split}
f(S,A,T,C_P,C_r,r_{\max}):=&\,\,T(T+1)r_{\max}((C_P+1)^{T+1}-1)\sqrt{S}\\
&\,+2C_r\dfrac{(C_P+1)^{T+2}-(T+2)C_P-1}{C_P^2}\sqrt{S}\\
&\, +S^{\frac{3}{2}}A(T+2)^3r_{\max}+\sqrt{SA(T+1)}+\sqrt{T}.
\end{split}
\end{equation*}
\end{theorem}
\begin{remark}
In the special case when the transition dynamics $P_t(s'|s,a,L_t)$ are independent of $L_t$, as generally assumed  in the literature of discrete-time mean-field games \cite{elie2020convergence, perrin2020fictitious, perolat2021scaling}, we have $C_P=0$ and a much tighter  constant 
\[
f(S,A,T,C_P,C_r,r_{\max})=C_r(T+1)(T+2)\sqrt{S}+S^{\frac{3}{2}}A(T+2)^3r_{\max}+\sqrt{SA(T+1)}+\sqrt{T}.
\]
 Note that, this tighter constant is polynomial in all the problem parameters.

\end{remark}

Conversely, 
one can characterize  an $\epsilon$-Nash equilibrium solution in terms of the sub-optimality of \eqref{mfvmo-full}, without the additional assumption of  the Lipschitz continuity  in Theorem \ref{VMO-vs-MFNE}.
\begin{theorem}\label{MFNE-vs-VMO}
Let $\pi=\{\pi_t\}_{t\in\mathcal{T}}$ be an $\epsilon$-Nash equilibrium solution, \ie, $\text{\rm Expl}(\pi)\\\leq\epsilon$. Define  $L=\Gamma(\pi)$, \[
y=[V_1^\star(L), V_2^\star(L), \dots,V_T^\star(L), -V_0^\star(L)]\in\mathbb{R}^{S(T+1)},
\]
and $z=[z_0,z_1,\dots,z_T]\in\mathbb{R}^{SA(T+1)}$ with $z_t=[z_t^1,\dots,z_t^A]$ ($t=0,\dots,T$), where
\[
z_t^a=V_t^\star(L)-r_t(\cdot,a,L_t)-P_t^a(L_t)V_{t+1}^\star(L)\in\mathbb{R}^S, \quad t=0,\dots,T-1,\,a\in\mathcal{A}
\]
and $z_T^a=V_T^\star(L)-r_T(\cdot,a,L_T)\in\mathbb{R}^S$ ($a\in\mathcal{A}$). Then $y,z,L$ is a feasible solution to \eqref{mfvmo-full} and $f^{\text{\rm MF-OMO}}(y,z,L)\leq \epsilon$. 
\end{theorem}

\section{Finding Nash equilibrium solutions via solving MF-OMO}\label{solve_mfvmo}

The optimization formulation \eqref{mfvmo-full} enables applications of  families of optimization algorithms to find  multiple Nash equilibrium solutions. In this section, we first present the projected gradient descent algorithm and establish its convergence guarantees to the Nash equilibrium solutions of general mean-field games. We then present a finite-time convergent algorithm for solving a special class of mean-field games with linear rewards and mean-field independent dynamics. 
More discussion on stochastic algorithms,  convergence to stationary points as well as practical tricks such as reparametrization, acceleration and variance reduction can be found in Appendix \ref{sgd_and_stationary_convergence}. 

To start, we will need the following  assumption. 
\begin{assumption}\label{C2smooth}
$P_t(s,a,L_t)$ and $r_t(s,a,L_t)$ are both second-order continuously differentiable in $L_t$ within some open set containing $\Delta(\mathcal{S}\times\mathcal{A})$ for any $s\in\mathcal{S},a\in\mathcal{A}$.  
\end{assumption}
This assumption is easily verifiable and holds for  numerical examples in many existing works including \cite{guo2020general,cui2021approximately} and for theoretical studies of  \cite{perrin2020fictitious,perolat2021scaling}.

For notation simplicity,  denote the concatenation of $y,z,L$ as $\theta$, and recall  the earlier notation for  feasible set of \eqref{mfvmo-full} as $\Theta$ (\cf \eqref{Theta_feasible}).
Then we see that Assumption \ref{C2smooth}, together with the compactness of the probability simplex of $L_t$ and the norm bounds on $y$ and $z$ in \eqref{mfvmo-full}, immediately leads to the following proposition.
\begin{proposition}\label{strong_smooth_fact}
Under Assumption \ref{C2smooth}, $f^{\text{\rm MF-OMO}}(\theta)$ is $M$-strongly smooth for some positive constant $M>0$,  for any $\theta\in\Theta$. That is, for any $\theta, \theta'\in\Theta$,  we have 
$\|\nabla_{\theta} f^{\text{\rm MF-OMO}}(\theta)-\nabla_{\theta} f^{\text{\rm MF-OMO}}(\theta')\|_2\leq M\|\theta-\theta'\|_2$.
\end{proposition}

In addition, Assumption \ref{C2smooth}, together with the compactness of $\Theta$, implies the continuity condition in Proposition \ref{existence} and the Lipschitz continuity condition in Theorem \ref{VMO-vs-MFNE}. 
Moreover, optimal solution(s) exist for \eqref{mfvmo-full} under Assumption \ref{C2smooth}. 
We will denote the optimal solution set of \eqref{mfvmo-full} 
as $\Theta^\star$ hereafter. 

\subsection{PGD and convergence to Nash equilibrium solutions}
\paragraph{Projected gradient descent (PGD).} The algorithm of projected gradient descent  updates $\theta$ with the following iterative process:
\begin{equation}\label{pgd_mfvmo}
    \theta_{k+1}=\theta_k-\eta_kG_{\eta_k}(\theta_k)=\textbf{Proj}_{\Theta}(\theta_k-\eta_k \nabla_{\theta}f^{\text{MF-OMO}}(\theta_k)).
\end{equation}
Here {\color{black}$\eta_k>0$ is the step-size of the $k$-th iteration for which appropriate ranges are specified below}, and the projection operator $\textbf{Proj}_{\Theta}$ projects $L$ to the probability simplex, 
$z$ to a specialized polyhedra $\{z\,|\,z\geq0, \|z\|_1\leq SA(T^2+T+2)r_{\max}\}$, 
and $y$ to the $\ell_2$-normed ball with norm $S(T+1)(T+2)r_{\max}/2$. All these projections can be efficiently evaluated in linear time and (almost) closed-form (\cf 
\cite{duchi2008efficient, liu2009efficient,songsiri2011projection,condat2016fast}
for projection onto the probability simplex and the specialized polyhedra, and \cite{parikh2014proximal} for the projection onto the $\ell_2$-normed balls). For simplicity, we always assume that the initialization $\theta_0$ is feasible, \ie, $\theta_0\in\Theta$.

\paragraph{Convergence to Nash equilibrium solutions.}
To find a Nash equilibrium solution of the mean-field game (\eqref{setup:MFG} and \eqref{consistency_pop_flow}) by  the PGD algorithm, we will need an additional definability assumption  stated below, {\color{black}which is one of the most commonly adopted assumptions in the existing convergence theory of nonconvex optimization.}
\begin{assumption}\label{definability}
For any $s\in\mathcal{S},a\in\mathcal{A}$, $P_t(s,a,L_t)$ and $r_t(s,a,L_t)$ (as functions of $L_t$) are both restrictions of definable functions on the log-exp structure\footnote{\color{black}For simplicity, below we say that a function is definable if it’s definable on the log-exp
structure.} to $\Delta(\mathcal{S}\times\mathcal{A})$.
\end{assumption}

Definable functions cover a broad class  of (both smooth and nonsmooth) functions, including all semialgebraic functions, all analytic functions on definable compact sets\footnote{\color{black}A set is definable if it can be defined by the image/range of a definable function.}, and the exponential and logarithm functions. Moreover, any finite combination of definable functions via summation, subtraction, product, affine mapping, composition, (generalized) inversion, max and min, partial supremum and partial infimum, as well as reciprocal (restricted to a compact domain) inside their domain of definitions are definable. {\color{black}Generally speaking, definable functions include all functions that are ``programmable'', such as those that can be defined in \texttt{NumPy}, \texttt{SciPy} and \texttt{PyTorch}.} The precise definition of definability for both sets and functions and the log-exp structure  can be found in Appendix \ref{definability_appendix} 
(see also \cite[Section 4.3]{attouch2010proximal} and \cite{coste2000introduction-ominimal}).

Under Assumptions \ref{C2smooth} and \ref{definability}, we will show that the PGD algorithm converges to Nash equilibrium solution(s) when the initialization is close. Moreover, if the initialization is close to some isolated Nash equilibrium solution (\eg, when there is a finite number of Nash equilibrium solutions), the iterates converge to that specific Nash equilibrium solution. As a result, different initializations will enable us to retrieve all possible Nash equilibrium solutions when there are finitely many of them. The proof is left to Appendix \ref{local_conv_global_proof}. 

\begin{theorem}[Convergence to Nash equilibrium solutions]\label{local_conv_global}
Under Assumptions \ref{C2smooth} and \ref{definability}, let $\theta_k=(y^k,z^k,L^k)$ {\rm(}$k\geq 0${\rm)} be the sequence generated by PGD \eqref{pgd_mfvmo} with $\eta_k=\eta\in(0,2/M)$. Then for any Nash equilibrium solution $(\pi^\star,L^\star)$, there exists $\epsilon_0>0$, such that for any $L^0\in\Delta(\mathcal{S}\times\mathcal{A})$ with $\|L^0-L^\star\|_1\leq \epsilon_0$, if we initialize PGD with $\theta_0=(y^0,z^0,L^0)$, where $y^0=[V_1^\star(L^0), V_2^\star(L^0), \dots,V_T^\star(L^0), -V_0^\star(L^0)]\in\mathbb{R}^{S(T+1)}$, $z^0=[z_{0,0},\dots,z_{0,T}]\in\mathbb{R}^{SA(T+1)}$, $z_{0,t}=[z_{0,t}^{1},\dots,z_{0,t}^{A}]$ {\rm(}$t\in\mathcal{T}${\rm)}, with 
\[
z_{0,t}^{a}=V_t^\star(L^0)-r_t(\cdot,a,L_t^0)-P_t^a(L_t^0)V_{t+1}^\star(L^0)\in\mathbb{R}^S,\quad t=0,\dots,T-1,\,a\in\mathcal{A}
\]
and 
$z_{0,T}^a=V_T^\star(L^0)-r_T(\cdot,a,L_T^0)\in\mathbb{R}^S$ {\rm(}$a\in\mathcal{A}${\rm)}, then
we have the following:
\begin{itemize}
    \item $\lim_{k\rightarrow\infty}{\bf dist}(\theta_k,\Theta^\star)=0$, and $\|L^k-L^\star\|_2\leq g(\epsilon_0)$ for any $k\geq 0$. Here 
    $g:\mathbb{R}_+\rightarrow\mathbb{R}_+$ is a function such that $\lim_{\epsilon\rightarrow0}g(\epsilon)=0$.\footnote{The closed-form definition of $g$ is left to Appendix \ref{local_conv_global_proof} 
    and can be found in \eqref{g_def} there.} 
    {\color{black}Here ${\bf dist}$ is the point-set distance defined as ${\bf dist}(\theta_k,\Theta^\star)=\inf_{\theta\in\Theta^\star}\|\theta_k-\theta\|_2$.}
    \item The limit point set of $\{\theta_k\}_{k\geq 0}$ is nonempty, and for any limit point $\bar{\theta}=(\bar{y},\bar{z},\bar{L})$ of the iteration sequence $\{\theta_k\}_{k\geq 0}$, any $\bar{\pi}\in\Pi(\bar{L})$ is a Nash equilibrium solution of  the mean-field game {\rm(}\eqref{setup:MFG} and \eqref{consistency_pop_flow}{\rm)}.  
    \item For any sequence of $\pi^k\in\Pi(L^k)$,  $\lim_{k\rightarrow\infty}\text{\rm Expl}(\pi^k)=0$.
\end{itemize} 

In addition, for any isolated Nash equilibrium solution $(\pi^\star,L^\star)$ (\ie, there exists $\epsilon$ such that any $(\pi,L)$ with $\|L-L^\star\|_2\leq \epsilon$ is not a Nash equilibrium solution), if we initialize PGD with $\theta_0=(y^0,z^0,L^0)$ in the same way as above, then $\lim_{k\rightarrow\infty}L^k=L^\star$ and $\lim_{k\rightarrow\infty}\text{\rm Expl}(\pi^k)=0$ for any sequence $\pi^k\in\Pi(L^k)$. 
\end{theorem}



\subsection{Finite-time solvability for a class of mean-field games} 
For mean-field games with linear rewards and mean-field independent dynamics,  we will show that a Nash equilibrium solution can be found in finite time under the MF-OMO framework. Such a class of mean-field games include the LR (left-right) and RPS (rock-paper-scissors) problems in \cite{cui2021approximately}.

More precisely, consider mean-field games with 
{\color{black}$r_t(s,a,L_t)=\bar{r}_{s,a,t}+L_t^\top\bar{R}_{s,a,t}$},  $P_t(s'|s,a,L_t)=\bar{P}_{s',s,a,t}$ for some constants $\bar{r}_{s,a,t},\,\bar{P}_{s',s,a,t}\in\mathbb{R}$ and constant vectors {\color{black}$\bar{R}_{s,a,t}\in\mathbb{R}^{SA}$}. 
Here, $A_L$ is independent of  $L$, hence  denoted as {\color{black}$\bar{A}$ (to distinguish from the number of actions $A$)}, and $c_L$ is of the form {\color{black}$c_L=\bar{c}+\bar{C} L$}, where $\bar{c}\in\mathbb{R}^{SA(T+1)}$ is the vectorization of {\color{black}$[-\bar{r}_{s,a,t}]_{s\in\mathcal{S},a\in\mathcal{A},t\in\mathcal{T}}\in \mathbb{R}^{S\times A\times (T+1)}$} (with column-major order), and {\color{black}$\bar{C}\in\mathbb{R}^{SA(T+1)\times SA(T+1)}$} is defined as 

\[\color{black}
\bar{C}=\textbf{diag}\left(
\left[\begin{matrix}
-R_{1,1,0}^\top\\
\vdots\\
-R_{S,A,0}^\top
\end{matrix}
\right],
\left[\begin{matrix}
-R_{1,1,1}^\top\\
\vdots\\
-R_{S,A,1}^\top
\end{matrix}
\right],\dots,
\left[\begin{matrix}
-R_{1,1,T}^\top\\
\vdots\\
-R_{S,A,T}^\top
\end{matrix}
\right]
\right).
\]
By \eqref{mf-vmo-constr}, one can see that finding Nash equilibrium solution(s) of this class of mean-field games is equivalent to solving the following linear complementarity problem (LCP):
    \begin{equation*}
    \begin{array}{ll}
    \text{\rm minimize}_{y,z,L} \quad &0\\
     \text{\rm subject to}  &{\color{black}\left[\begin{array}{ll}
         0 & \bar{A} \\
         \bar{A}^\top & -\bar{C}
     \end{array}\right]}\left[\begin{array}{l}
          y  \\
          L 
     \end{array}\right]+\left[\begin{array}{l}
          0  \\
          z 
     \end{array}\right]=\left[\begin{array}{l}
          b  \\
          \bar{c} 
     \end{array}\right],\\
        & z^\top L=0,\quad L\geq 0,\quad z\geq 0.
        \end{array}
    \end{equation*}
Notice that for any solution $(y,z,L)$, we have either $z_i=0$ or $L_i=0$ for any $i=1,\dots,SA(T+1)$. Hence it suffices to consider linear programs induced by fixing $z_{\mathcal{D}}=0$ and $L_{\mathcal{D}^c}=0$ for some subset $\mathcal{D}$ of $\{1,\dots,SA(T+1)\}$, which leads to a total of $2^{SA(T+1)}$ linear programs. Since a linear program can be solved to its global optimality in finite time (\eg, by simplex algorithms \cite{martin2012large}), we have the following proposition. 
\begin{proposition}\label{finite-convergence-linear-mfg}
Suppose that the mean-field game {\rm(}\eqref{setup:MFG} and \eqref{consistency_pop_flow}{\rm)}   has linear rewards and mean-field independent dynamics. Then its Nash equilibrium solution can be found in finite time. 
\end{proposition}
In practice, to solve the LCP efficiently by exploiting the linear complementarity structure, one may resort to pivoting procedures \cite{murty1997linear} or Lemke's algorithm \cite{adler2011linear}.

\section{Proofs of main results}\label{sec:proofs}

\subsection{Proof of Lemma \ref{condition_A_B_prime}}\label{subsec:proof-lemma3}
For any Nash equilibrium solution $(\pi,L)$, denote
$d_t^\star(s,a)=\mathbb{P}^{\pi,L}(s_t=s,a_t=a)$ for the probability distribution of any representative agent taking policy $\pi$ under
$L$.
We will first show that $\pi\in\Pi(d^\star)$ and $(d^\star,L)$ satisfies conditions (A) and (B).
Since $(\pi,L)$ is a Nash equilibrium solution, it satisfies \eqref{setup:MFG}. By Lemma \ref{prop:MDP_LP} and by considering the MDP for the representative agent, $\pi\in\Pi(d^\star)$ and condition (A) is satisfied for $d^\star$ and $L$. Condition (B) can be proved by induction. When $t=0$, $d_0^\star(s,a)=\mathbb{P}^{\pi,L}(s_0=s,a_0=a)=\mu_0\pi_0(a|s)=L_0(s,a)$ for any $s\in\mathcal{S}$ and any $a\in\mathcal{A}$, since \eqref{consistency_explicit_1} holds for $L_0$. Suppose $d_t^\star(s,a)=L_t(s,a)$ holds for all $t\leq t'$ and all $s\in\mathcal{S}$ and $a\in\mathcal{A}$. Then for any $s\in\mathcal{S}$ and any $a\in\mathcal{A}$, 
\begin{align*}
    d_{t'+1}^\star(s,a)&=\mathbb{P}^{\pi,L}(s_{t'+1}=s,a_{t'+1}=a)=\mathbb{P}^{\pi,L}(s_{t'+1}=s)\pi_{t'+1}(a|s)\\
    &=\pi_{t'+1}(a|s)\sum_{s'\in\mathcal{S}}\sum_{a'\in\mathcal{A}}\mathbb{P}^{\pi,L}(s_{t'}=s',a_{t'}=a',s_{t'+1}=s)\\
    &=\pi_{t'+1}(a|s)\sum_{s'\in\mathcal{S}}\sum_{a'\in\mathcal{A}}\mathbb{P}^{\pi,L}(s_{t'}=s',a_{t'}=a')P_{t'}(s|s',a',L_{t'})\\
    &=\pi_{t'+1}(a|s)\sum_{s'\in\mathcal{S}}\sum_{a'\in\mathcal{A}}d_{t'}(s',a')P_{t'}(s|s',a',L_{t'})\\
    &=\pi_{t'+1}(a|s)\sum_{s'\in\mathcal{S}}\sum_{a'\in\mathcal{A}}L_{t'}(s',a')P_{t'}(s|s',a',L_{t'})=L_{t'+1}(s,a),
\end{align*}
where the last equation holds by \eqref{consistency_explicit_2}. 

On the other hand, for any $(d^\star,L)$ satisfying conditions (A) and (B), let $\pi_t(a|s)=\frac{L_t(s,a)}{\sum_{a'\in\mathcal{A}}L_t(s,a')}$ when $\sum_{a'\in\mathcal{A}}L_t(s,a')>0$ and let $\pi_t(\cdot|s)$ be any probability vector if $\sum_{a'\in\mathcal{A}}L_t(s,a')=0$, we will show that $(\pi,L)$ is the Nash equilibrium solution of the mean-field game (\eqref{setup:MFG} and \eqref{consistency_pop_flow}). 

By Lemma \ref{prop:MDP_LP}, condition  \eqref{setup:MFG} holds. For condition \eqref{consistency_pop_flow}, since $d^\star=L$, $L$ satisfies \eqref{cons_occ} by replacing $d$ with $L$. Therefore, for any $s'\in\mathcal{S},t=0,1,\dots,T-1$, 
\begin{align*}
    \sum_{s\in\mathcal{S}}\sum_{a\in\mathcal{A}}L_t(s,a)P_t(s'|s,a,L_t)=\sum_{a\in\mathcal{A}}L_{t+1}(s',a).
\end{align*}
Multiplying both sides with $\pi_{t+1}(a'|s')$, and by the definition of $\pi$,  
\begin{align*}
 \pi_{t+1}(a'|s')\sum_{s\in\mathcal{S}}\sum_{a\in\mathcal{A}}L_t(a|s)P_t(s'|s,a,L_t)= L_{t+1}(s',a').  
\end{align*}
Hence condition \eqref{consistency_pop_flow}.\qed

\subsection{Proof of Proposition \ref{interpretation_y_z}}
Since $y,z,L$ is a solution to \eqref{mf-vmo-constr}, $L\geq 0$, {\color{black}$\mathbf{1}^\top L_t=1$, $t\in\mathcal{T}$} (by $A_LL=b$), \ie, $L\in\Delta(\mathcal{S}\times\mathcal{A})$. Note that we use $\mathbf{1}$ to denote the all-one vector (with appropriate dimensions).  To show that $\hat{y},\hat{z},L$ solves \eqref{mf-vmo-constr}, it reduces to show that $A_L^\top \hat{y}+\hat{z}=c_L$, $\hat{z}\geq 0$ and $\hat{z}^\top L=0$. 

By writing $\hat{y}$ as $\hat{y}=[\hat{y}_0,\dots,\hat{y}_T]$ with $\hat{y}_t\in\mathbb{R}^S$, $\hat{z}$ as $\hat{z}=[\hat{z}_0,\dots,\hat{z}_T]$ with $\hat{z}_t\in\mathbb{R}^{SA}$, we have 
\[
A_L^\top \hat{y}+\hat{z}=
\left[\begin{array}{c}
W_0(L_0)^\top \hat{y}_0+Z^\top \hat{y}_T+\hat{z}_0,\\
W_1(L_1)^\top \hat{y}_1-Z^\top \hat{y}_0+\hat{z}_1,\\
\vdots\\
W_{T-1}(L_{T-1})^\top \hat{y}_{T-1}-Z^\top \hat{y}_{T-2}+\hat{z}_{T-1},\\
-Z^\top \hat{y}_{T-1}+\hat{z}_T.
\end{array}\right],
\]
which, by the definition of $W_t(L_t)$ and $Z$, becomes 
\[
A_L^\top \hat{y}+\hat{z}=
\left[\begin{array}{c}
P_0^1(L_0)\hat{y}_0+\hat{y}_T+\hat{z}_0^1\\
\vdots\\
P_0^A(L_0)\hat{y}_0+\hat{y}_T+\hat{z}_0^A\\
P_1^1(L_1)\hat{y}_1-\hat{y}_0+\hat{z}_1^1\\
\vdots\\
P_1^A(L_1)\hat{y}_1-\hat{y}_0+\hat{z}_1^A\\
\vdots\\
P_{T-1}^1(L_{T-1})\hat{y}_{T-1}-\hat{y}_{T-2}+\hat{z}_{T-1}^1\\
\vdots\\
P_{T-1}^A(L_{T-1})\hat{y}_{T-1}-\hat{y}_{T-2}+\hat{z}_{T-1}^A\\
-\hat{y}_{T-1}+\hat{z}_T^1\\
\vdots\\
-\hat{y}_{T-1}+\hat{z}_T^A
\end{array}\right].
\]
Now noticing that $\hat{y}_T=-V_0^\star(L)$, $\hat{y}_t=V_{t+1}^\star(L)$, 
$\hat{z}_T^a=V_T^\star(L)-r_T(\cdot,a,L_T)$, 
and $\hat{z}_t^a=V_t^\star(L)-r_t(\cdot,a,L_t)-P_t^aV_{t+1}^\star(L)$ for $a\in\mathcal{A}$, $t=0,\dots,T-1$, we see
\[
A_L^\top \hat{y}+\hat{z}=[-r_0(\cdot,\cdot,L_0), \dots,-r_T(\cdot,\cdot,L_T)]=c_L.
\]

Moreover, we have $V_T^\star(L)=\max_{a\in\mathcal{A}}r_T(\cdot,a,L_T)$, and that 
\[
V_t^\star(L)=\max_{a\in\mathcal{A}}\left\{r_t(\cdot,a,L_t)+P_t^a(L_t)V_{t+1}^\star(L)\right\},
\]
implying that $\hat{z}\geq0$ by its definition. Note that these are the Bellman equations for the mean-field induced MDP $\mathcal{M}(L)$.  

It remains to prove that $\hat{z}^\top L=0$. To see this, we first show that $\hat{y}_{t-1}\leq y_{t-1}$ for $t=1,\dots,T$ and $\hat{y}_{-1}\geq y_{-1}$ by backward induction on $t$ from $T$ to $0$. Here $\hat{y}_{-1}$ and $y_{-1}$ are defined as $\hat{y}_T$ and $y_T$, respectively.

Firstly, note that $A_L^\top y+z=c_L$ and $z\geq 0$. Hence similar to the expansions above, we have 
\[
z_T^a=y_{T-1}-r(\cdot,a,L_T),\quad z_0^a=-y_T-r_0(\cdot,a,L_0)-P_0^a(L_0)y_0,\quad \forall a\in\mathcal{A}, 
\]
and 
\[
z_t^a=y_{t-1}-r_t(\cdot,a,L_t)-P_t^a(L_t)y_t, \quad \forall a\in\mathcal{A},\, t=0,\dots,T-1.
\]
Now for the base step, since $z_T^a\geq 0$ for all $a\in\mathcal{A}$, we have 
\[
y_{T-1}\geq \max_{a\in\mathcal{A}}r(\cdot,a,L_T)=V_T^\star(L)=\hat{y}_{T-1}.
\]
For the induction step, suppose that $\hat{y}_{t}\leq y_{t}$ for some $0<t\leq T-1$. Then since $z_{t}^a\geq 0$ for all $a\in\mathcal{A}$, we have 
\begin{equation*}
\begin{split}
y_{t-1}&\geq \max_{a\in\mathcal{A}}\left\{r_{t}(\cdot,a,L_{t})+P_{t}^a(L_{t})y_{t}\right\}\geq \max_{a\in\mathcal{A}}\left\{r_t(\cdot,a,L_t)+P_t^a(L_t)\hat{y}_t\right\}\\
&= \max_{a\in\mathcal{A}}\left\{r_t(\cdot,a,L_t)+P_t^a(L_t)V_{t+1}^\star(L)\right\}=V_t^\star(L)=\hat{y}_{t-1}.
\end{split}
\end{equation*}
Finally, suppose that $\hat{y}_0\leq y_0$. Then since $z_0^a\geq 0$ for all $a\in\mathcal{A}$, 
\begin{equation*}
\begin{split}
y_{-1}=y_T&\leq \min_{a\in\mathcal{A}}\left\{-r_0(\cdot,a,L_0)-P_0^a(L_0)y_0\right\}\leq \min_{a\in\mathcal{A}}\left\{-r_0(\cdot,a,L_0)-P_0^a(L_0)\hat{y}_0\right\}\\
&=-\max_{a\in\mathcal{A}}\left\{r_0(\cdot,a,L_0)+P_0^a(L_0)V_1^\star(L)\right\}=-V_0^\star(L)=\hat{y}_T=\hat{y}_{-1}.
\end{split}
\end{equation*}
Therefore $\hat{y}_T\geq y_T$.

Now, since $y,z,L$ is a solution to problem \eqref{mf-vmo-constr}, we have 
\[
b^\top y = (A_LL)^\top y+L^\top z = L^\top(A_L^\top y+z)=L^\top c_L=c_L^\top L.
\]
And for any $y',z'$ satisfying $A_{L}^\top y'+z'=c_{L}$ and $z'\geq 0$, since $A_{L}L=b$ and $L\geq 0$, we have 
\begin{equation}\label{comp_gap_y_z}
c_{L}^\top L - b^\top y'=(A_{L}^\top y'+z')^\top L-L^\top A_{L}^\top y'=L^\top z'\geq 0.
\end{equation}
In particular, $b^\top y=c_L^\top L\geq b^\top y'$ for any $y'$ such that $A_{L}^\top y'+z'=c_{L}$ for some $z'\geq 0$. That is, $y$ (together with $z$) is the solution to the following linear optimization problem for fixed $L$:
\begin{equation}\label{YOPT}
\begin{array}{ll}
\text{maximize}_{y',z'} & b^\top y'\\
\text{subject to} & A_L^\top y'+z'=c_L, \quad z'\geq 0.
\end{array}
\end{equation}

Now taking $y'=\hat{y}$ and $z'=\hat{z}$, we see that $A_L^\top \hat{y}+\hat{z}=c_L$ and $\hat{z}\geq 0$ as proved above. Hence $\hat{y},\hat{z}$ is a feasible solution to problem \eqref{YOPT}, and $b^\top y\geq b^\top \hat{y}$. Since $\hat{y}_T\geq y_T$, 
\[
b^\top\hat{y}=\mu_0^\top \hat{y}_T\geq \mu_0^\top y_T= b^\top y.
\]
Combined, $b^\top y=b^\top \hat{y}$, hence by \eqref{comp_gap_y_z},
\[
L^\top \hat{z}=c_L^\top L-b^\top \hat{y}=c_L^\top L-b^\top y = 0. 
\] 
\begin{flushright}
\qed
\end{flushright}

\subsection{Proof of Theorem \ref{VMO-vs-MFNE}}
It is established via the following perturbation lemma for general MDPs. 
\begin{lemma}\label{perturbation}
Consider two finite MDPs\footnote{Finite MDP refers to MDPs with finite state and action spaces.} $\mathcal{M}^1$ and $\mathcal{M}^2$ with finite horizon, with their respective transition probabilities $\{p^1_t\}_{t\in\mathcal{T}}$ and $\{p^2_t\}_{t\in\mathcal{T}}$, rewards $\{r_t^1\}_{t\in\mathcal{T}}$ and $\{r_t^2\}_{t\in\mathcal{T}}$, expected total rewards $V^{1,\pi}$ and $V^{2,\pi}$ for any policy $\pi$. In addition, denote their corresponding optimal values as $V^{1,\star}$ and $V^{2,\star}$. Then we have 
\[
|V^{1,\pi}-V^{2,\pi}|\leq \dfrac{T(T+1)r_{\max}}{2}\|p^1-p^2\|_{\infty,1}+\|r^1-r^2\|_{1,\infty},
\]
and
\[
|V^{1,\star}-V^{2,\star}|\leq \dfrac{T(T+1)r_{\max}}{2}\|p^1-p^2\|_{\infty,1}+\|r^1-r^2\|_{1,\infty}.
\]
Here $r_{\max}$ is such that $|r_t^1(s,a)|\leq r_{\max}$ and $|r_t^2(s,a)|\leq r_{\max}$ for any $s\in\mathcal{S},a\in\mathcal{A}$. 
\end{lemma}

\paragraph{Proof of Lemma \ref{perturbation}.}
The proof consists of two parts.

\textit{Part I: Values for fixed policy $\pi$.}
Let $V_t^{i,\pi}(s)$ be the total expected reward 
starting from state $s$ at time $t$ under policy $\pi$ in the MDP $\mathcal{M}^i$ ($i=1,2$). Clearly $V^{i,\pi}=\sum_{s\in\mathcal{S}}\mu_0(s)V_0^{i,\pi}(s)$ ($i=1,2$). 

Now by the dynamic programming principle, we have for $i=1,2$, $t=0,\dots,T-1$ and $s\in\mathcal{S}$, 
\[
V_t^{i,\pi}(s)=\sum_{a\in\mathcal{A}}\pi_t(a|s)r_t^i(s,a)+\sum_{a\in\mathcal{A}}\sum_{s'\in\mathcal{S}}\pi_t(a|s)p_t^i(s'|s,a)V_{t+1}^{i,\pi}(s'),
\]
and $V_T^{i,\pi}(s)=\sum_{a\in\mathcal{A}}\pi_t(a|s)r_T^i(s,a)$, 
implying 
\begin{equation}\label{recursive_bd}
\begin{split}
|V_t^{1,\star}(s)&-V_t^{2,\star}(s)|\leq\, \sum_{a\in\mathcal{A}}|r_t^1(s,a)-r_t^2(s,a)|\pi_t(a|s)\\
&+\sum_{a\in\mathcal{A}}\sum_{s'\in\mathcal{S}}\pi_t(a|s)\left|p_t^1(s'|s,a)V_{t+1}^{1,\pi}(s')-p_t^2(s'|s,a)V_{t+1}^{2,\pi}(s')\right|\\
\leq&\, \max_{a\in\mathcal{A}}\sum_{s'\in\mathcal{S}}|p_t^1(s'|s,a)-p_t^2(s'|s,a)||V_{t+1}^{1,\pi}(s')|\\
&+\sum_{a\in\mathcal{A}}\sum_{s'\in\mathcal{S}}\pi_t(a|s)p_t^2(s'|s,a)|V_{t+1}^{1,\pi}(s')-V_{t+1}^{2,\pi}(s')|\\
&+\max_{a\in\mathcal{A}}|r_t^1(s,a)-r_t^2(s,a)|\\
(a)\leq&\,(T-t)r_{\max}\max_{s\in\mathcal{S},a\in\mathcal{A}}\sum_{s'\in\mathcal{S}}|p_t^1(s'|s,a)-p_t^2(s'|s,a)|+\|V_{t+1}^{1,\pi}-V_{t+1}^{2,\pi}\|_{\infty}\\
&+\max_{s\in\mathcal{S},a\in\mathcal{A}}|r_t^1(s,a)-r_t^2(s,a)|.
\end{split}
\end{equation}
Here ($a$) uses the fact that $|V_{t+1}^{i,\pi}(s)|\leq (T-t)r_{\max}$, and $V_t^{i,\pi}$ is viewed as a vector in $\mathbb{R}^{S}$ when taking the $\ell_\infty$ norm. 

By taking max over $s\in\mathcal{S}$ on the left-hand side of \eqref{recursive_bd} and telescoping the inequalities, we have 
\begin{equation*}
\begin{split}
\|V_0^{1,\pi}-V_0^{2,\pi}\|_{\infty}\leq& \sum_{t=0}^{T-1}(T-t)r_{\max}\max_{s\in\mathcal{S},a\in\mathcal{A}}\sum_{s'\in\mathcal{S}}|p_t^1(s'|s,a)-p_t^2(s'|s,a)|\\
&+\sum_{t=0}^T\max_{s\in\mathcal{S},a\in\mathcal{A}}|r_t^1(s,a)-r_t^2(s,a)|\\
\leq& \dfrac{T(T+1)r_{\max}}{2}\|p^1-p^2\|_{\infty,1}+\|r^1-r^2\|_{1,\infty}.
\end{split}
\end{equation*}
Here we use the fact that \[
V_{T}^{1,\pi}(s)-V_{T}^{2,\pi}(s)=\sum_{a\in\mathcal{A}}r_T^1(s,a)\pi_T(a|s)-\sum_{a\in\mathcal{A}}r_T^2(s,a)\pi_T(a|s).
\]


\textit{Part II: Optimal values.}
Let $V_t^{i,\star}(s)$ be the  value function for MDP $\mathcal{M}_i$ ($i=1,2$) starting at time $t\in\mathcal{T}$. 
Then we have $V^{i,\star}=\sum_{s\in\mathcal{S}}\mu_0(s)V_0^{i,\star}(s)$ ($i=1,2$). 

Now by the dynamic programming principle, we have for $i=1,2$, $t=0,\dots,T-1$ and $s\in\mathcal{S}$, 
\[
V_t^{i,\star}(s)=\max_{a\in\mathcal{A}}\left(r_t^i(s,a)+\sum_{s'\in\mathcal{S}}p_t^i(s'|s,a)V_{t+1}^{i,\star}(s')\right),
\]
which implies that 
\begin{equation}\label{recursive_bd_2}
\begin{split}
|V_t^{1,\star}(s)&-V_t^{2,\star}(s)|\\
&\leq \max_{a\in\mathcal{A}}\left\{|r_t^1(s,a)-r_t^2(s,a)|+\sum_{s'\in\mathcal{S}}\left|p_t^1(s'|s,a)V_{t+1}^{1,\star}(s')-p_t^2(s'|s,a)V_{t+1}^{2,\star}(s')\right|\right\}\\
&\leq \max_{a\in\mathcal{A}}\sum_{s'\in\mathcal{S}}|p_t^1(s'|s,a)-p_t^2(s'|s,a)|V_{t+1}^{1,\star}(s')\\
&\quad+\max_{a\in\mathcal{A}}\sum_{s'\in\mathcal{S}}p_t^2(s'|s,a)|V_{t+1}^{1,\star}(s')-V_{t+1}^{2,\star}(s')|+\max_{a\in\mathcal{A}}|r_t^1(s,a)-r_t^2(s,a)|\\
(a)&\leq(T-t)r_{\max}\max_{s\in\mathcal{S},a\in\mathcal{A}}\sum_{s'\in\mathcal{S}}|p_t^1(s'|s,a)-p_t^2(s'|s,a)|+\|V_{t+1}^{1,\star}-V_{t+1}^{2,\star}\|_{\infty}\\
&\quad+\max_{s\in\mathcal{S},a\in\mathcal{A}}|r_t^1(s,a)-r_t^2(s,a)|.
\end{split}
\end{equation}
Here 
 in ($a$) we use the fact that $V_{t+1}^{i,\star}(s)\in[0,(T-t)r_{\max}]$, and $V_t^{i,\star}$ is viewed as a vector in $\mathbb{R}^{S}$ when taking the $\ell_\infty$ norm. 

By taking max over $s\in\mathcal{S}$ on the left-hand side of \eqref{recursive_bd_2} and telescoping the inequalities, we have 
\begin{equation*}
\begin{split}
\|V_0^{1,\star}-V_0^{2,\star}\|_{\infty}\leq& \sum_{t=0}^{T-1}(T-t)r_{\max}\max_{s\in\mathcal{S},a\in\mathcal{A}}\sum_{s'\in\mathcal{S}}|p_t^1(s'|s,a)-p_t^2(s'|s,a)|\\
&+\sum_{t=0}^T\max_{s\in\mathcal{S},a\in\mathcal{A}}|r_t^1(s,a)-r_t^2(s,a)|\\
\leq& \dfrac{T(T+1)r_{\max}}{2}\|p^1-p^2\|_{\infty,1}+\|r^1-r^2\|_{1,\infty}.
\end{split}
\end{equation*}
Here we use the fact that $V_{T}^{1,\star}(s)-V_{T}^{2,\star}(s)=r_T^1(s,a)-r_T^2(s,a)$. \qed 

\paragraph{Proof of Theorem \ref{VMO-vs-MFNE}.}

We first define
$d=\tilde{\Gamma}(\pi,L)$ for the $\pi\in\Pi(L)$, where $\tilde{\Gamma}(\pi,L)$ is recursively defined by $\tilde{\Gamma}(\pi,L)_0(s,a):=\mu_0(s)\pi_0(a|s)$ (for any $ s\in\mathcal{S},a\in\mathcal{A}$) and 
\[
\tilde{\Gamma}(\pi,L)_{t+1}(s,a):=\pi_{t+1}(a|s)\sum_{s'\in\mathcal{S}}\sum_{a'\in\mathcal{A}}\tilde{\Gamma}(\pi,L)_t(s',a')P_t(s|s',a',L_t),\quad \forall s\in\mathcal{S},a\in\mathcal{A}.
\]
Then $d$ is the occupation measure of policy $\pi$ for MDP $\mathcal{M}(L)$. 

\medskip

\noindent\textit{Step 1: Closeness between $d$ and $L$.}
By definition, $A_Ld=b$. 
We first prove by induction that for all $t$,
\begin{equation}\label{ineq:induction}
  \sum_{s\in\mathcal{S}}\sum_{a\in\mathcal{A}}|d_t(s,a)-L_t(s,a)|\leq \sqrt{S}(t+1)\epsilon.  
\end{equation}
When $t=0$, 
\begin{equation}\label{induction:init}
    \begin{split}
        \sum_{s\in\mathcal{S}}\sum_{a\in\mathcal{A}}|d_0(s,a)-L_0(s,a)|&=\sum_{s\in\mathcal{S}}\sum_{a\in\mathcal{A}}\pi_0(a|s)|\mu_0(s)-\sum_{a'\in\mathcal{A}}L_0(s,a')|\\
        &=\sum_{s\in\mathcal{S}}|\mu_0(s)-\sum_{a'\in\mathcal{A}}L_0(s,a')|\leq \sqrt{S}\epsilon,
    \end{split}
\end{equation}
where the last inequality is by $\|A_LL-b\|_2\leq \epsilon$. 
Now suppose \eqref{ineq:induction} holds for $t$. 
Then by the construction of $\pi$ and $d$,
\begin{equation}\label{d_induced_consistency}
    d_{t+1}(s,a)=\pi_{t+1}(a|s)\sum_{s'\in\mathcal{S}}\sum_{a'\in\mathcal{A}}d_t(s',a')P_t(s|s',a',L_t),
\end{equation}
\begin{equation}
    L_{t+1}(s,a)=\pi_{t+1}(a|s)\sum_{a'\in\mathcal{A}}L_{t+1}(s,a').
\end{equation}
Therefore 
\begin{equation}\label{eq:bound-d-L}
    \begin{split}
        \sum_{s\in\mathcal{S}}&\sum_{a\in\mathcal{A}}|d_{t+1}(s,a)-L_{t+1}(s,a)|\\
        &=\sum_{s\in\mathcal{S}}\sum_{a\in\mathcal{A}}\pi_{t+1}(a|s)|\sum_{s'\in\mathcal{S}}\sum_{a'\in\mathcal{A}}d_t(s',a')P_t(s|s',a',L_t)-\sum_{a'\in\mathcal{A}}L_{t+1}(s,a')|\\
        &=\sum_{s\in\mathcal{S}}|\sum_{s'\in\mathcal{S}}\sum_{a'\in\mathcal{A}}d_t(s',a')P_t(s|s',a',L_t)-\sum_{a'\in\mathcal{A}}L_{t+1}(s,a')|\\
        &\leq \sum_{s\in\mathcal{S}}|\sum_{s'\in\mathcal{S}}\sum_{a'\in\mathcal{A}}\left(d_t(s',a')-L_t(s',a')\right)P_t(s|s',a',L_t)|\\
        &\qquad\qquad +\sum_{s\in\mathcal{S}}|\sum_{s'\in\mathcal{S}}\sum_{a'\in\mathcal{A}}L_t(s',a')P_t(s|s',a',L_t)-\sum_{a'\in\mathcal{A}}L_{t+1}(s,a')|\\
        &\leq \sqrt{S}(t+1)\epsilon + \sqrt{S}\epsilon=\sqrt{S}(t+2)\epsilon,
    \end{split}
\end{equation}
where the last inequality holds by using induction on the first term and the constraint on $L$ in $\|A_LL-b\|_2\leq \epsilon$. This completes the induction.

The above bound \eqref{eq:bound-d-L}, together with the fact that $0\leq z^\top L\leq \epsilon^2$, immediately implies that 
\begin{equation*}
\begin{split}
|z^\top d|&\leq |z^\top (L-d)|+|z^\top L|\leq \|z\|_1\|L-d\|_{\infty}+\epsilon^2\\
&\leq SA(T^2+T+2)r_{\max}\sqrt{S}(T+1)\epsilon+\epsilon^2.
\end{split}
\end{equation*}
Hence by $A_Ld=b$ and $\|A_L^\top y+z-c_L\|_2^2\leq \epsilon^2$, we have 
\begin{equation*}
\begin{split}
|b^\top y-c_L^\top d|&=|d^\top A_L^\top y-c_L^\top d|\leq |d^\top z|+|d^\top (A_L^\top y+z-c_L)|\\
&\leq SA(T^2+T+2)r_{\max}\sqrt{S}(T+1)\epsilon+\epsilon^2+\|d\|_2\|A_L^\top y+z-c_L\|_2\\
&\leq SA(T^2+T+2)r_{\max}\sqrt{S}(T+1)\epsilon+\epsilon^2 + \sqrt{T}\epsilon.
\end{split}
\end{equation*}
Here we use $\|d_t\|_1=1$ for any $t\in\mathcal{T}$ and $\|d_t\|_2\leq \|d_t\|_1$. 

\medskip

\noindent\textit{Step 2: Near-optimality of $d$ (and $\pi$) for $\mathcal{M}(L)$.} 
Now we show that $d$ is near-optimal for the MDP $\mathcal{M}(L)$. To see this, define $\Delta:=A_L^\top y+z-c_L$. Here $\Delta$ can be viewed both as a vector in $\mathbb{R}^{SA(T+1)}$ and a sequence of $T+1$ matrices $\{\Delta_t\}_{t\in\mathcal{T}}$ with $\Delta_t\in\mathbb{R}^{S\times A}$. Then we have $\|\Delta\|_2^2\leq \epsilon^2$. 

Consider the following linear program with variables $\hat{y},\hat{z}$:
\begin{equation}\label{yz_Delta}
\begin{array}{ll}
\text{maximize}_{\hat{y},\hat{z}} & b^\top\hat{y}\\
\text{subject to} & A_L^\top\hat{y}+\hat{z}=c_L+\Delta,\quad \hat{z}\geq 0,
\end{array}
\end{equation}
and denote $\hat{y}^{\star}, \hat{z}^\star$ as its optimal solution. Then since $y,z$ is its feasible solution, we have $b^\top y\leq b^\top \hat{y}^\star$. 

Note that the dual problem of \eqref{yz_Delta} can be written as
\begin{equation}\label{d_Delta}
\begin{array}{ll}
\text{minimize}_{\hat{d}} & (c_L+\Delta)^\top\hat{d}\\
\text{subject to} & A_L\hat{d}=b,\quad \hat{d}\geq 0.
\end{array}
\end{equation}
By Lemma \ref{prop:MDP_LP} and \eqref{cond_A_rewrite}, if $\hat{d}^\star$ solves \eqref{d_Delta}, then any $\hat{\pi}^\star\in\Pi(\hat{d}^\star)$ is an optimal policy for the MDP $\hat{\mathcal{M}}(L)$ with transitions $P_t(s'|s,a,L_t)$ and rewards $r_t(s,a,L_t)-\Delta_t(s,a)$ (with $L=\{L_t\}_{t\in\mathcal{T}}$ fixed), and the optimal value of the objective function of \eqref{d_Delta} is equal to the negative optimal value of the MDP $\hat{\mathcal{M}}(L)$. By the strong duality of linear programs, we see that $b^\top \hat{y}^\star=(c_L+\Delta)^\top \hat{d}^\star$ is equal to the negative optimal value of the MDP $\hat{\mathcal{M}}(L)$. Note that the components of $c_L$ are negative rewards, and that the existence of the optimal solution $\hat{d}^\star$ follows from the existence of optimal policies for any finite MDP with finite horizon and Lemma \ref{prop:MDP_LP} and \eqref{cond_A_rewrite}. 

Now, let $d^\star$ be an optimal solution to the following linear program:
\begin{equation}\label{d_exact}
\begin{array}{ll}
\text{minimize}_{d} & c_L^\top d\\
\text{subject to} & A_L d=b,\quad d\geq 0.
\end{array}
\end{equation}
Then again by Lemma \ref{prop:MDP_LP} and \eqref{cond_A_rewrite}, $c_L^\top d^\star$ is equal to the negative optimal value of the MDP $\mathcal{M}(L)$. Hence by applying Lemma \ref{perturbation} to the MDPs $\hat{\mathcal{M}}(L)$ and $\mathcal{M}(L)$, 
\begin{equation*}
\begin{split}
|c_L^\top d^\star - (c_L+\Delta)^\top \hat{d}^\star|&\leq \|\Delta\|_{1,\infty}= \sum_{t=0}^T\max_{s\in\mathcal{S},a\in\mathcal{A}}|\Delta_t(s,a)|\leq \sum_{t\in\mathcal{T},s\in\mathcal{S},a\in\mathcal{A}}|\Delta_t(s,a)|\\
&\leq \sqrt{SA(T+1)}\epsilon.
\end{split}
\end{equation*}

Putting the above bounds together, we conclude that 
\begin{equation}\label{cd_gap}
\begin{split}
c_L^\top d-c_L^\top d^\star&=c_L^\top d-b^\top y+b^\top y-c_L^\top d^\star\\
&\leq c_L^\top d-b^\top y+b^\top \hat{y}^\star-c_L^\top d^\star\\
&= c_L^\top d-b^\top y+(c_L+\Delta)^\top \hat{d}^\star-c_L^\top d^\star\\
&\leq SA(T^2+T+2)r_{\max}\sqrt{S}(T+1)\epsilon+\epsilon^2 + \sqrt{T}\epsilon+\sqrt{SA(T+1)}\epsilon.
\end{split}
\end{equation}
Since $d$ is a feasible solution to \eqref{d_exact},  again by Lemma \ref{prop:MDP_LP}, \eqref{cd_gap} implies that $\pi$, by definition satisfying $\pi\in\Pi(d)$ from \eqref{d_induced_consistency}, is a $(C_{S,A,T}\epsilon+\epsilon^2)$-suboptimal solution to the MDP $\mathcal{M}(L)$, where 
$C_{S,A,T}=S^{\frac{3}{2}}A(T+2)^3r_{\max}+\sqrt{SA(T+1)}+\sqrt{T}$.

\medskip

\noindent\textit{Step 3: Bounding the gap between $L$ and $\Gamma(\pi)$.}
We then bound the difference between $L$ and $\Gamma(\pi)$ by proving the following for all $t$:
\begin{equation}\label{ineq:induction2}
    \sum_{s\in\mathcal{S}}\sum_{a\in\mathcal{A}}|\Gamma(\pi)_{t}(s,a)-L_{t}(s,a)|\leq \frac{(C_P+1)^{t+1}-1}{C_P}\sqrt{S}\epsilon.
\end{equation}
This is  again proved by induction. When $t=0$, it holds by noticing $\Gamma(\pi)_0(s,a)=d_0(s,a)$ and \eqref{induction:init}. Now suppose \eqref{ineq:induction2} holds for $t$, consider when $t+1$, by the definition of $\Gamma(\pi)$, we have 
\begin{equation}
    \Gamma(\pi)_{t+1}(s,a)=\pi_{t+1}(a|s)\sum_{s'\in\mathcal{S}}\sum_{a'\in\mathcal{A}}\Gamma(\pi)_t(s',a')P_t(s|s',a',\Gamma(\pi)_t).
\end{equation}
Then 
\begin{equation}
    \begin{split}
        \sum_{s\in\mathcal{S}}&\sum_{a\in\mathcal{A}}|\Gamma(\pi)_{t+1}(s,a)-L_{t+1}(s,a)|\\
        &=\sum_{s\in\mathcal{S}}|\sum_{s'\in\mathcal{S}}\sum_{a'\in\mathcal{A}}\Gamma(\pi)_t(s',a')P_t(s|s',a',\Gamma(\pi)_t)-\sum_{a'\in\mathcal{A}}L_{t+1}(s,a')|\\
        &\leq \sum_{s\in\mathcal{S}}|\sum_{s'\in\mathcal{S}}\sum_{a'\in\mathcal{A}}\Gamma(\pi)_t(s',a')(P_t(s|s',a',\Gamma(\pi)_t)-P_t(s|s',a',L_t)|\\
        &\qquad \qquad +\sum_{s\in\mathcal{S}}\sum_{s'\in\mathcal{S}}\sum_{a'\in\mathcal{A}}|\Gamma(\pi)_t(s',a')-L_t(s',a')|P_t(s|s',a',L_t)\\
        &\qquad\qquad +\sum_{s\in\mathcal{S}}|\sum_{s'\in\mathcal{S}}\sum_{a'\in\mathcal{A}}L_t(s',a')P_t(s|s',a',L_t)-\sum_{a'\in\mathcal{A}}L_{t+1}(s,a')|\\
        &\leq \sum_{s'\in\mathcal{S}}\sum_{a'\in\mathcal{A}}\Gamma(\pi)_{t}(s',a')\sum_{s\in\mathcal{S}}|P_t(s|s',a',\Gamma(\pi)_t)-P_t(s|s',a',L_t)|\\
        &\qquad\qquad + \sum_{s'\in\mathcal{S}}\sum_{a'\in\mathcal{A}}|\Gamma(\pi)_t(s',a')-L_t(s',a')|+\sqrt{S}\epsilon\\
        &\leq (C_P+1)\sum_{s'\in\mathcal{S}}\sum_{a'\in\mathcal{A}}|\Gamma(\pi)_t(s',a')-L_t(s',a')|+\sqrt{S}\epsilon\\
        &\leq (C_P+1)\frac{(C_P+1)^{t+1}-1}{C_P}\sqrt{S}\epsilon+\sqrt{S}\epsilon=\frac{(C_P+1)^{t+2}-1}{C_P}\sqrt{S}\epsilon.
    \end{split}
\end{equation}
This concludes the induction. Consequently, 
\begin{equation*}
\begin{split}
\|P^{\Gamma(\pi)}-P^{L}\|_{\infty,1}&=\max_{t\in\mathcal{T}}\|P_t^{\Gamma(\pi)}-P_t^{L}\|_{\infty,1}\leq \max_{t\in\mathcal{T}}C_P\|\Gamma(\pi)_t-L_t\|_1\\
&\leq ((C_P+1)^{T+1}-1)\sqrt{S}\epsilon,
\end{split}
\end{equation*}
and similarly, 
\begin{equation*}
\begin{split}
\|r^{\Gamma(\pi)}-r^L\|_{1,\infty}&=\sum_{t=0}^T\|r_t^{\Gamma(\pi)}-r_t^L\|_\infty\leq C_r \sum_{t=0}^T\|\Gamma(\pi)_t-L_t\|_1\\
&\leq C_r\dfrac{(C_P+1)^{T+2}-(T+2)C_P-1}{C_P^2}\sqrt{S}\epsilon.
\end{split}
\end{equation*}

Finally, by applying Lemma \ref{perturbation} to  the MDPs $\mathcal{M}(L)$ and $\mathcal{M}(\Gamma(\pi))$,  and utilizing the fact that $\pi$ is $(C_{S,A,T}\epsilon+\epsilon^2)$-suboptimal for the MDP $\mathcal{M}(L)$, we see that 
\begin{equation*}
\begin{split}
\text{Expl}(\pi)&=V_{\mu_0}^\star(\Gamma(\pi))-V_{\mu_0}^{\pi}(\Gamma(\pi)) \\
&=V_{\mu_0}^\star(\Gamma(\pi))-V_{\mu_0}^\star(L)+V_{\mu_0}^\star(L) - V_{\mu_0}^{\pi}(L)+V_{\mu_0}^{\pi}(L)-V_{\mu_0}^{\pi}(\Gamma(\pi))\\
&\leq T(T+1)r_{\max}((C_P+1)^{T+1}-1)\sqrt{S}\epsilon\\
&\quad +2C_r\dfrac{(C_P+1)^{T+2}-(T+2)C_P-1}{C_P^2}\sqrt{S}\epsilon +C_{S,A,T}\epsilon+\epsilon^2
\end{split}
\end{equation*}
Here $V_{\mu_0}^\star(L)=\max_{\pi'}V_{\mu_0}^{\pi'}(L)=\sum_{s\in\mathcal{S}}\mu_0(s)[V_0^{\star}(L)]_s$ is the optimal value of the MDP $\mathcal{M}(L)$ for any mean-field flow $L$. \qed




\subsection{Proof of Theorem \ref{MFNE-vs-VMO}}
The feasibility of $y,z,L$ is obvious from the definition. In particular, $L=\Gamma(\pi)$ implies that  $L\geq 0$ and {\color{black}$\mathbf{1}^\top L_t=1$, $t\in\mathcal{T}$}, and the same argument leading to Corollary \ref{y_z_bound} implies that $\|y\|_2\leq S(T+1)(T+2)r_{\max}/2$ and $\mathbf{1}^\top z\leq SA(T^2+T+2)r_{\max}$, and the proof of Proposition \ref{interpretation_y_z} and in particular, the Bellman optimality shows that $z\geq 0$. 

Furthermore, again by the construction and the proof of Proposition \ref{interpretation_y_z}, we also have $A_L L=b$ by taking $L=\Gamma(\pi)$ and $A_L^\top y+z=c_L$. Hence 
\[
f^{\text{MF-OMO}}(y,z,L)=\|A_LL-b\|_2^2+\|A_L^\top y+z-c_L\|_2^2+z^\top L=z^\top L.
\]
It remains to show that $z^\top L\leq \epsilon$. 
To see this, notice that 
\[
z^\top L=(c_L-A_L^\top y)^\top L=c_L^\top L-y^\top A_LL=c_L^\top L-b^\top y.
\]
Now recall that $L=\Gamma(\pi)$ implies that $L_t(s,a)=\mathbb{P}^{\pi,L}(s_t=s,a_t=a)$ for $t\in\mathcal{T}$ by the proof of Lemma \ref{condition_A_B_prime}. Hence
\begin{equation*}
\begin{split}
c_L^\top L&=-\sum_{s\in\mathcal{S},a\in\mathcal{A},t\in\mathcal{T}}r_t(s,a,L_t)L_t(s,a)\\
&=-\sum_{s\in\mathcal{S},a\in\mathcal{A},t\in\mathcal{T}}r_t(s,a,L_t)\mathbb{P}^{\pi,L}(s_t=s,a_t=a)=-V_{\mu_0}^\pi(L).
\end{split}
\end{equation*}
Combined with the fact that $b^\top y=-\sum_{s\in\mathcal{S}}\mu_0(s)[V_0^\star(L)]_s=-V_{\mu_0}^\star(L)$, we see that 
\[
z^\top L=c_L^\top L-b^\top y=V_{\mu_0}^\star(L)-V_{\mu_0}^{\pi}(L)=V_{\mu_0}^\star(\Gamma(\pi))-V_{\mu_0}^{\pi}(\Gamma(\pi))=\text{Expl}(\pi)\leq \epsilon.
\] \qed

\section{Extensions}\label{refined_mfne}
The MF-OMO framework can be  extended to other variants of mean-field games. This section details its extension to \textit{personalized} mean-field games; and its extension to multi-population mean-field games is straightforward and omitted here. 

Personalized mean-field games are mean-field games involving non-homogeneous players, which can be found in many applications. 
In such problems, every player is associated with some information (type) which characterizes the heterogeneity among players. 
Personalized mean-field games generalize mean-field games by incorporating the information (type) into the state of players, with heterogeneity in the initial distributions. To deal with the heterogeneity, one needs to define a ``stronger'' Nash equilibrium solution for these personalized mean-field games, where the policy sequence $\pi$ is optimal from any initial state given the mean-field flow. We say that the policy sequence $\pi=\{\pi_t\}_{t\in\mathcal{T}}$ and  a mean-field flow $L=\{L_t\}_{t\in\mathcal{T}}$ constitute a refined Nash equilibrium solution of the finite-time horizon personalized mean-field game, if the following conditions are satisfied.




\begin{definition}[Refined Nash equilibrium solution]\label{mfne-gen}
~
\begin{itemize}
    \item[1)] (Optimality) Fixing $\{L_t\}_{t\in\mathcal{T}}$, for any initial state $s\in\mathcal{S}$, $\{\pi_t\}_{t\in\mathcal{T}}$ solves the following optimization problem:
\begin{equation}\label{optimality_agent_refined}
\begin{array}{ll}
\text{maximize}_{\{\pi_t\}_{t\in\mathcal{T}}} & \mathbb{E}[\sum_{t=0}^Tr_t(s_t,a_t,L_t)|s_0=s]\\
\text{subject to} & s_{t+1}\sim P_t(\cdot|s_t,a_t,L_t),\, a_t\sim\pi_t(\cdot|s_t),\, t\in\mathcal{T}\backslash \{T\}, 
\end{array}  
\end{equation}
\textit{i.e.}, $\{\pi_t\}_{t\in\mathcal{T}}$ is optimal for the representative agent given the mean-field flow $\{L_t\}_{t\in\mathcal{T}}$; 
\item[2)] (Consistency) Fixing $\{\pi_t\}_{t\in\mathcal{T}}$, the consistency of mean-field flow holds, \ie, 
\begin{equation}\label{consistency_pop_flow_refined}
\begin{split}
&L_t=\mathbb{P}_{s_t,a_t},\\
&\text{where } s_{t+1}\sim P_t(\cdot|s_t,a_t,L_t),\, a_t\sim\pi_t(\cdot|s_t),\, s_0\sim \mu_0,\, t\in\mathcal{T}\backslash \{T\}. 
\end{split}
\end{equation}
\end{itemize} 
\end{definition}

For personalized mean-field games, we introduce a modified version of exploitability. Instead of simply averaging value functions over $\mu_0$ (\cf \eqref{def:expl}), here the averaging is over a uniform initial distribution, which corresponds to the arbitrariness of the initial state in \eqref{optimality_agent_refined}. The exploitability is defined as follows:
\[
\text{Expl}^{\text{refined}}(\pi):=\dfrac{1}{S}\max_{\pi'}\sum_{s\in\mathcal{S}}\left([V_0^{\pi'}(\Gamma(\pi))]_s-[V_0^{\pi}(\Gamma(\pi))]_s\right).
\]
As in Section \ref{MFG_opt_idea}, $\pi$ is a refined Nash equilibrium solution if and only if  $\text{Expl}^{\text{refined}}(\pi)=0$. 

To find the refined Nash equilibrium solution, we need to find the ``refined'' optimal policy for the mean-field induced MDP as described in  \eqref{optimality_agent_refined}, which is optimal under any initial state $s_0=s\in\mathcal{S}$ (instead of an initial state with distribution $\mu_0$ as in \eqref{setup:MFG}).
The proposition below characterizes the linear program formulation for \eqref{optimality_agent_refined}, 
a counterpart of Lemma \ref{prop:MDP_LP}. For brevity,  for any $\nu_0\in\Delta(\mathcal{S})$, we denote $\mathbb{P}^{\pi,L,\nu_0}(\cdot)$ for the probability distribution of any representative agent taking policy $\pi$ under the mean-field flow 
$L$, \ie, the state and/or action distribution generated by $s_0\sim\nu_0$, $s_{t+1}\sim P_t(\cdot|s_t,a_t,L_t)$, $a_t\sim\pi_t(\cdot|s_t)$ for $t=0,\dots,T-1$. 
\begin{proposition}\label{prop:MDP_LP_gen}
Fix $\{L_t\}_{t\in\mathcal{T}}$. Suppose that $\{\pi_t\}_{t\in\mathcal{T}}$ is an $\epsilon$-suboptimal policy for \eqref{optimality_agent_refined} for any $s\in\mathcal{S}$. 
Define $\nu_0(s)=\frac{1}{S}$ for all $s\in\mathcal{S}$, $\nu_t(s):=\mathbb{P}^{\pi,L,\nu_0}(s_t=s)$, 
and $d_t(s,a):=\nu_t(s)\pi_t(a|s)$ for $s\in\mathcal{S}$ and $a\in\mathcal{A}$. Then $\pi\in\Pi(d)$ and  $d=\{d_t\}_{t\in\mathcal{T}}$ is a feasible $\epsilon$-suboptimal solution to the following linear program:
\begin{equation}\label{lq:personalized}
  \begin{array}{ll}
      \text{maximize}_{d}\quad &\sum\limits_{t\in\mathcal{T}}\sum\limits_{s\in\mathcal{S}}\sum\limits_{a\in\mathcal{A}} d_t(s,a)r_t(s,a,L_t) \\
      \text{subject to} \quad &\sum\limits_{s\in\mathcal{S}}\sum\limits_{a\in\mathcal{A}}d_t(s,a)P_t(s'|s,a,L_t)=\sum\limits_{a\in\mathcal{A}}d_{t+1}(s',a),\\
      &\hspace{4.05cm}\forall s'\in\mathcal{S},t\in\mathcal{T}\backslash\{T\},\\
     &\sum\limits_{a\in\mathcal{A}}d_0(s,a)=\frac{1}{S}, \quad \forall s\in\mathcal{S},\\
      &d_t(s,a)\geq 0,\quad \forall s\in\mathcal{S},a\in\mathcal{A},t\in\mathcal{T}.
    \end{array}
\end{equation}
Conversely, suppose that $d=\{d_t\}_{t\in\mathcal{T}}$ is a feasible $\epsilon$-suboptimal solution to the above linear program \eqref{lq:personalized}. Then for any $\pi\in\Pi(d)$, 
$\{\pi_t\}_{t\in\mathcal{T}}$ is an $S\epsilon$-suboptimal policy for \eqref{optimality_agent_refined} for any $s\in\mathcal{S}$. In addition, the optimal value of the objective function of the above linear program \eqref{lq:personalized} is equal to the average optimal value of \eqref{optimality_agent_refined} over $s\in\mathcal{S}$. 
\end{proposition}
Note that instead of the specified initial distribution $\mu_0$ in Lemma \ref{prop:MDP_LP}, Proposition \ref{prop:MDP_LP_gen} chooses a special initial distribution where all components are equal to $\frac{1}{S}$. The proof is a direct combination of Lemma \ref{prop:MDP_LP} and the following classic result of MDP.
\begin{lemma}
If $\pi$ is an optimal policy for the MDP $\mathcal{M}$ with initial distribution $\nu_0$ satisfying $\nu_0(s)>0$ for all $s\in\mathcal{S}$, then $\pi$ is an optimal policy for the MDP $\mathcal{M}$ with any initial state and any initial distribution.
\end{lemma}

Then similar to Theorem \ref{thm:nash-opt} one can establish the corresponding equivalence between refined Nash equilibrium solutions and optimal solutions to a feasibility optimization problem. This is derived by replacing the $\mu_0$ in $b$ in Theorem \ref{thm:nash-opt} with $\frac{1}{S}{\bf 1}\in\mathbb{R}^{S}$. More precisely, based on the above result, one can characterize the optimality conditions of $\pi$ given $L$, the counterpart of condition (A) with $\mu_0$ replaced by the uniform distribution over $\mathcal{S}$. Then combined with the consistency condition for the mean-field flow \eqref{consistency_pop_flow_refined}, one can establish the corresponding equivalence between refined Nash equilibrium solutions and optimal solutions to the following feasibility optimization problem: 
    \begin{equation}\label{refined_mfne_opt}
    \begin{array}{ll}
    \text{\rm minimize}_{\pi,y,z,d,L} \quad &0\\
     \text{\rm subject to}  &A_Ld=b',\quad A_L^\top y+z=c_L,\\
        & z^\top d=0,\quad d\geq 0,\quad z\geq 0,\\
        &\pi_t(a|s)\sum_{a'\in\mathcal{A}}d_t(s,a')=d_t(s,a), \, \forall s\in\mathcal{S},a\in\mathcal{A},t\in\mathcal{T}\\
        & L_0(s,a)=\mu_0(s)\pi_0(a|s),\quad \forall s\in\mathcal{S},a\in\mathcal{A},\\
        & L_{t+1}(s',a')=\pi_{t+1}(a'|s')\sum\limits_{s\in\mathcal{S}}\sum\limits_{a\in\mathcal{A}}L_t(s,a)P_t(s'|s,a,L_t), \\
        & \hspace{5.34cm} \forall s'\in\mathcal{S},t\in\mathcal{T}\backslash \{T\},\\
        & \sum_{a\in\mathcal{A}}L_0(s,a)=\mu_0(s), \quad \forall s\in\mathcal{S},\quad L\geq 0,
        \end{array}
    \end{equation}
    where $A_L,c_L$ are set as \eqref{kkt_param1} and \eqref{kkt_param2}, and $b'=[0,\dots,0,\frac{1}{S}{\bf 1}]\in\mathbb{R}^{S(T+1)}$. More precisely, $(\pi,L)$ is a refined Nash equilibrium solution if and only if $(\pi,y,z,d,L)$ is a solution to \eqref{refined_mfne_opt} for some $d,y,z$. 

Note that here the  condition (A) may not hold, therefore the consistency condition on mean-field flow can not be implied by $d=L$. As a result, the policy $\pi$ is directly included as part of the variables. 

In addition, noticing that  $L_t(s,a)=\sum_{a'\in\mathcal{A}}L_t(s,a')\pi_t(a|s)$ after summing over $a'$ on both sides of the fifth row of the constraints in \eqref{refined_mfne_opt},
and denoting $b=[0^\top,\dots,0^\top,\mu_0^\top]^\top\in\mathbb{R}^{S(T+1)}$, 
then the optimization problem \eqref{refined_mfne_opt} can 
be further simplified, hence the following theorem. 
\begin{theorem}\label{thm:nash-opt-gen}
Solving refined Nash equilibrium solution(s) of the personalized mean-field game {\rm(}\eqref{optimality_agent_refined} and \eqref{consistency_pop_flow_refined}{\rm)} is equivalent to solving the following feasibility optimization problem:
    \begin{equation}\label{refined_mfne_opt_simp}
    \begin{array}{ll}
    \text{\rm minimize}_{y,z,d,L} \quad &0\\
     \text{\rm subject to}  &A_Ld=b',\quad A_L^\top y+z=c_L,\\
        & z^\top d=0,\quad d\geq 0,\quad z\geq 0,\\
        & A_LL=b, \quad L\geq 0, \\
        & L_t(s,a)\sum\nolimits_{a'\in\mathcal{A}}d_t(s,a')=d_t(s,a)\sum\nolimits_{a'\in\mathcal{A}}L_t(s,a'),\\
        &\hspace{4.38cm} \forall s\in\mathcal{S},a\in\mathcal{A},t\in\mathcal{T}.
        \end{array}
    \end{equation}
    where $A_L,c_L$ are set as \eqref{kkt_param1} and \eqref{kkt_param2}, and $b'=[0,\dots,0,\frac{1}{S}{\bf 1}]\in\mathbb{R}^{S(T+1)}$. More precisely, if $(\pi,L)$ is a refined Nash equilibrium solution, then $(y,z,d,L)$ is a solution to \eqref{refined_mfne_opt_simp} for some $y,z,d$. Conversely, if $(y,z,d,L)$ is a solution to \eqref{refined_mfne_opt_simp}, then for any $\pi\in\Pi(d)$, $(\pi,L)$ constitutes a refined Nash equilibrium solution.
\end{theorem}
Note that the last row of constraints in \eqref{refined_mfne_opt_simp} is equivalent to requiring that $\Pi(d)=\Pi(L)$, where the equality is in terms of sets.

{\color{black}
\section{Numerical experiments}
In this section, we present numerical results of  algorithms based on the MF-OMO optimization framework. In the first part, MF-OMO is compared with the existing state-of-the-art algorithms for solving MFGs on the SIS problem introduced in \cite{cui2021approximately}. In the second part, it is shown that different initializations of MF-OMO algorithms converge to different NEs in the example introduced in Section \ref{prob_setup}.

\subsection{Comparison with existing algorithms}
In this section, we focus on the SIS problem introduced in \cite{cui2021approximately}. There are a large number of agents who choose to either social distance or to go out at each time step. Susceptible (S) agents who choose to go out may get infected (I), with probability proportional to the number of infected agents, while susceptible agents opt for social distancing will stay healthy. In the meantime, infected agents recover with certain probability at each time step. Agents in the system aim at finding the best strategy to minimize their costs induced from social distancing and being infected. The  parameters remain the same as in \cite{cui2021approximately}. 

Figures \ref{fig:sis_50} and \ref{fig:sis_100} show the comparisons between MF-OMO and existing algorithms on this SIS problems with $T=50$ and $T=100$, respectively. Here, MF-OMO is tested with different optimization algorithms including PGD (\cf \S\ref{solve_mfvmo}), Adam \cite{kingma2014adam}, and NAdam \cite{dozat2016incorporating} against online mirror descent (OMD) \cite{perolat2021scaling}, fictitious play (FP) \cite{perrin2020fictitious}, and prior descent (PD) \cite{cui2021approximately} with different choices of hyperparameters.\footnote{In the figure legends, the numbers after the OMD and FP algorithms are the step-sizes/learning-rates of these algorithms, while the number pairs after the PD algorithms are the temperatures and inner iterations of PD, respectively.} To make fair comparisons,  the maximum number of iterations for each algorithm are set differently so that the total runtime are comparable and that MF-OMO algorithms are given the smallest budget in terms of both the total number of iterations and the total runtime.  
With the same 
uniform initialization, it is clear from the figures that MF-OMO with Adam outperforms the remaining algorithms. The normalized exploitability\footnote{Normalized exploitability is defined to be exploitability devided by the initial exploitability.} of MF-OMO Adam quickly drops below $10^{-2}$ in terms of both the number of iterations and the runtime. All algorithms except for three variants of MF-OMO fail to achieve $10^{-2}$ for both $T=50$ and $T=100$.

\begin{figure}[h]
\centering
\begin{subfigure}{0.48\textwidth}
\includegraphics[width=2.46in, height=1.6in]{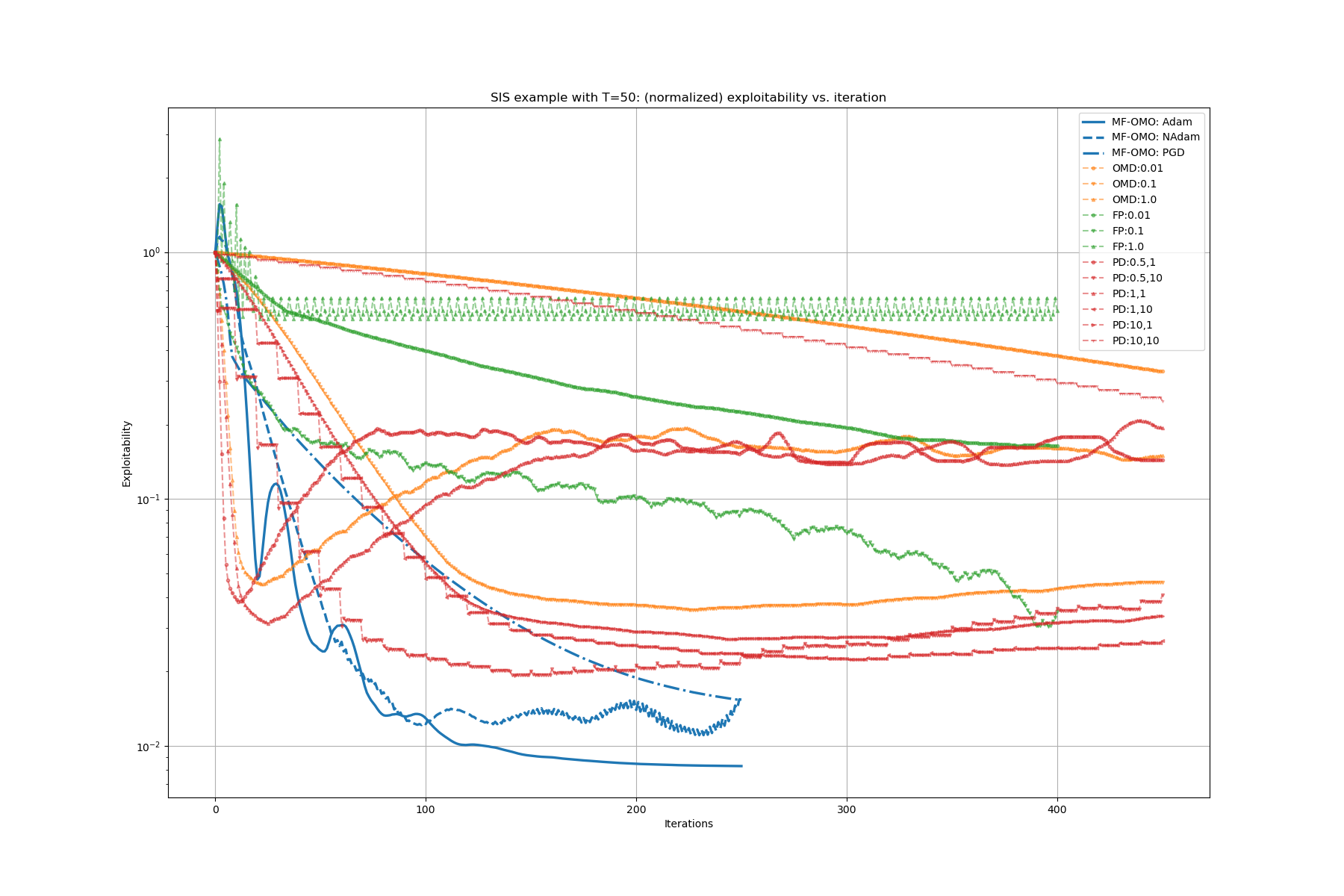}
\caption{Convergence against iterations}
\label{fig:sis_iter_50} 
\end{subfigure}
\begin{subfigure}{0.48\textwidth}
\includegraphics[width=2.46in, height=1.6in]{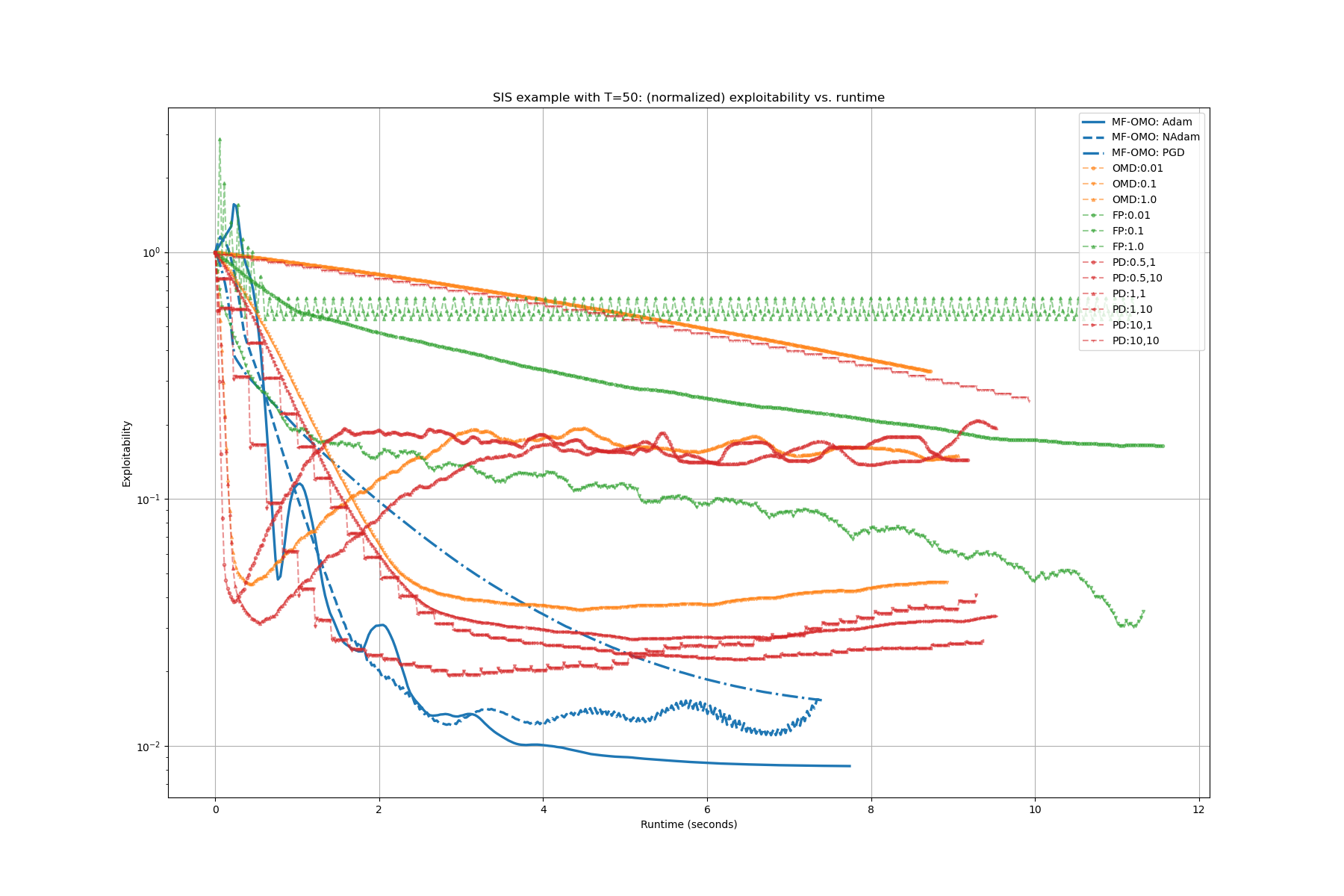}
\caption{Convergence against runtime}
\label{fig:sis_time_50} 
\end{subfigure}
\caption{Comparison of different algorithms on SIS problem with $T=50$.}
\label{fig:sis_50}
\end{figure}

\begin{figure}[h]
\centering
\begin{subfigure}{0.48\textwidth}
\includegraphics[width=2.46in, height=1.6in]{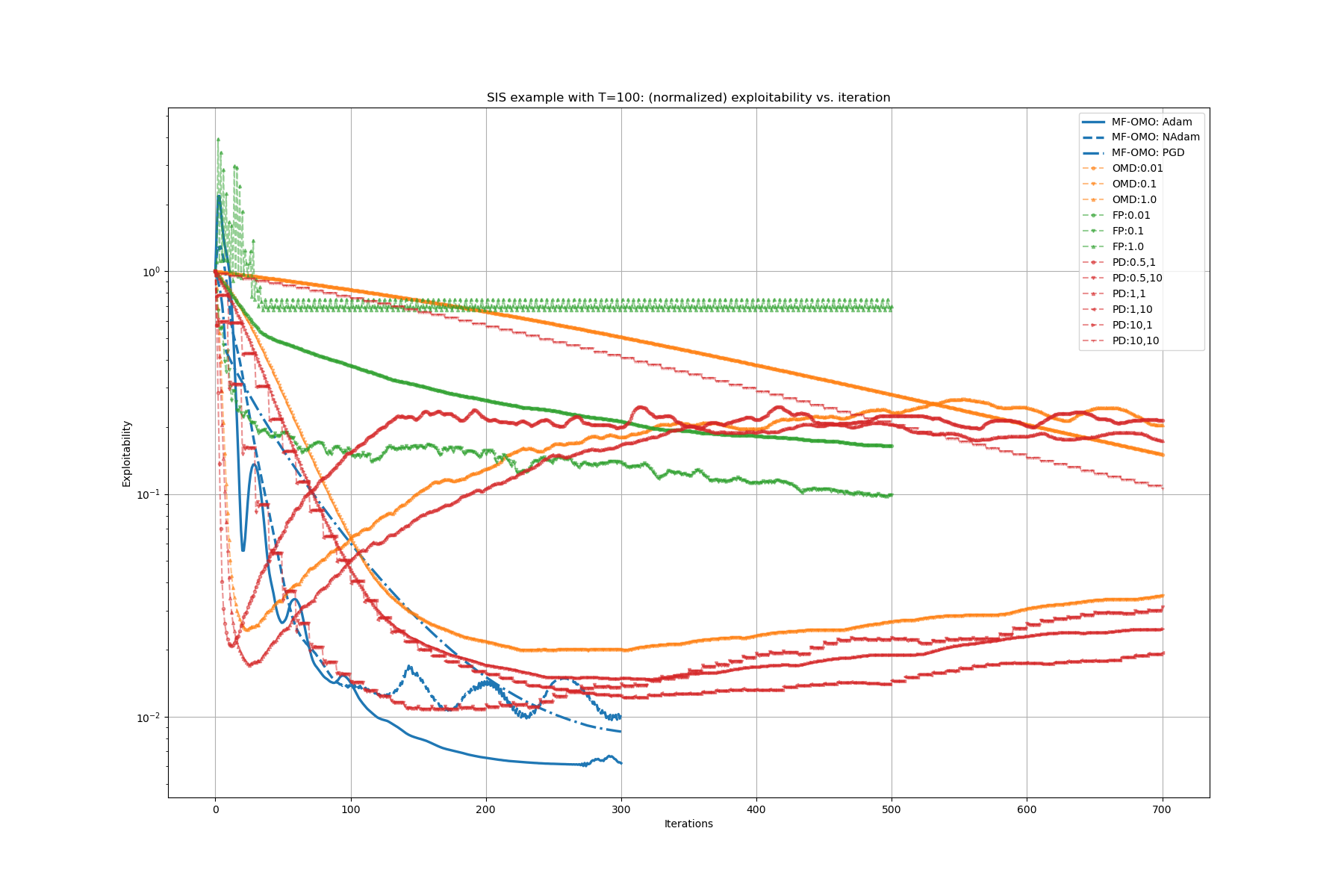}
\caption{Convergence against iterations}
\label{fig:sis_iter_100} 
\end{subfigure}
\begin{subfigure}{0.48\textwidth}
\includegraphics[width=2.46in, height=1.6in]{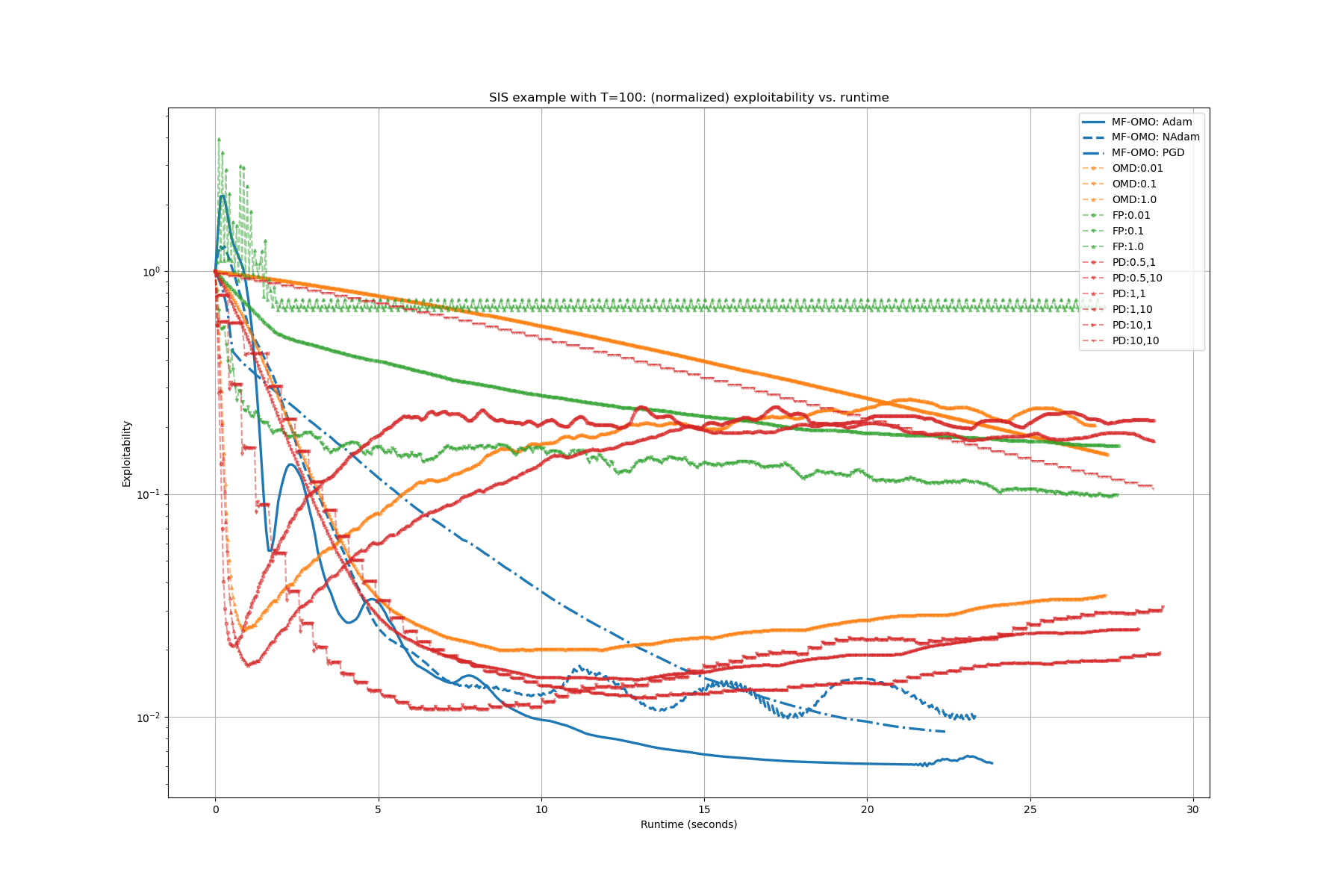}
\caption{Convergence against runtime}
\label{fig:sis_time_100} 
\end{subfigure}
\caption{Comparison of different algorithms on SIS problem with $T=100$.}
\label{fig:sis_100}
\end{figure}

\subsection{Multiple NEs}
In this section, we focus on the problem introduced in Section \ref{prob_setup} and show how algorithms based on the MF-OMO framework converge with different initializations. We choose $S=A=n=5$, $T=10$, $r^1=r^2=r^3=1.5$; and $r^4,\,r^5$, and $\{C_t\}_{t=1}^{10}$ are all independently sampled uniformly between $0$ and $1$. 

From discussions in Section \ref{prob_setup}, there are at least 3 different NEs, which correspond to agents staying in state 1, 2 and 3, respectively. Here, several different sets of random initializations are tested, denoted by RI$(i,\epsilon)$, which represents initializations randomly generated in the $\epsilon$ neighborhood of the $i$-th NE $(i=1,2,3)$. Each RI$(i,\epsilon)$ set consists of 20 independent samples, and the convergence behavior of MF-OMO NAdam algorithm is recorded in Table \ref{table:diff init}.\footnote{We choose NAdam as it converges significantly faster than other algorithms such as PGD and Adam for this specific problem, and our primary goal  is to study the behavior after the convergence.} For each neighborhood of size $\epsilon$ and NE $i$, the convergence behavior in RI$(i,\epsilon)$ is categorized using the following criteria: $p_0$ counts the proportion of samples whose normalized exploitabilities do not drop to $10^{-3}$ after 400 iterations; $p_1$ counts the proportion of samples that converge to an (approximate) NE (with normalized exploitability below $10^{-3}$) closest to $i$-th NE among the three NEs; $p_2$ counts the proportion of samples that converge to an (approximate) NE closer to the other two NEs.

\begin{table}
\centering
  \begin{tabular}{c|ccc|ccc|ccc}
    \hline\hline
    \multirow{2}{*}{$\epsilon$} &
      \multicolumn{3}{c|}{NE 1} &
      \multicolumn{3}{c|}{NE 2} &
      \multicolumn{3}{c}{NE 3} \\
    & $p_0$ & $p_1$ & $p_2$ & $p_0$ & $p_1$ & $p_2$ & $p_0$ & $p_1$ & $p_2$ \\
    \hline\hline
    0.05 & 0\% & 70\% & 30\% & 5\% & 75\% & 20\% & 0\% & 65\% & 35\% \\
    \hline
    0.2 & 0\% & 50\% & 50\% & 5\% & 65\% & 30\% & 5\% & 60\% & 35\% \\
    \hline
    0.8 & 5\% & 45\% & 50\% & 5\% & 40\% & 55\% & 0\% & 30\% & 70\% \\
    \hline
    1.0 & 0\% & 40\% & 60\% & 0\% & 45\% & 55\% & 0\% & 25\% & 75\% \\
    \hline\hline
  \end{tabular}
  \caption{Convergence behavior of different initializations}
  \label{table:diff init}
\end{table}

From Table \ref{table:diff init}, it is clear that most of the samples converge to some approximate NE with exploitabilities below the target $10^{-3}$ tolerance, regardless of the initializations. In addition, different initializations lead to different solutions. Specifically, when initializations are close to some NE, it is more likely for MF-OMO NAdam to converge to that specific NE, which is consistent with our theoretical study. Finally, as the neighborhood sizes $\epsilon$ increase, the convergence behavior become expectedly more chaotic, \textit{i.e.}, less concentrated on the center NE. 

}

\bibliographystyle{plain}
\bibliography{mfopt_arxiv_v3}

\newpage
\appendix
\section*{Appendix}
\addcontentsline{toc}{section}{Appendix}

\section{Additional optimization algorithms and convergence to stationary points}\label{sgd_and_stationary_convergence}
\subsection{Stochastic projected gradient descent (SPGD)}  
When the state space, action space and the time horizon are large, a single gradient evaluation in the PGD update \eqref{pgd_mfvmo} can be costly. 
To address this issue, we consider a stochastic variant of PGD (SPGD), which is suitable to handle problems with large $S$, $A$ and $T$. To invoke SPGD, first recall the explicitly expansion of \eqref{mfvmo-full} as \eqref{expansion-mfvmo}. 
The objective in \eqref{expansion-mfvmo} is a sum of $n:=S+ST+SA+SA(T-1)+SA+SA(T+1)=S(T+1)(2A+1)$ terms, which we denote as $f_i^{\text{MF-OMO}}(\theta)$ ($i=1,\dots,n$) for brevity. 
SPGD replaces the exact gradient $\nabla_{\theta}f^{\text{MF-OMO}}(\theta_k)$ in PGD with a mini-batch estimator of the following form
\begin{equation}\label{hatgk}
\hat{g}_k=\frac{\color{black}n}{|\mathcal{B}_k|}\sum_{i\in \mathcal{B}_k}\nabla_{\theta}f_i^{\text{MF-OMO}}(\theta_k),
\end{equation}
where the mini-batch $\mathcal{B}_k$ is a subset of $\{1,\dots,n\}$ sampled uniformly at random (without replacement) and independently across $k$ (\ie, iterations). Such a sampling approach ensures that $\hat{g}_k$ is an unbiased estimator of $\nabla_{\theta}f^{\text{MF-OMO}}(\theta_k)$ (conditioned on $\theta_k$). 

SPGD then proceeds by updating $\theta$ with the following iterative process:
\begin{equation}\label{spgd_mfvmo}
\theta_{k+1}=\textbf{Proj}_{\Theta}(\theta_k-\eta_k\hat{g}_k).
\end{equation}
Again, we assume that $\theta_0\in\Theta$.

\begin{proposition}\label{hatgk_properties}
Under Assumption \ref{C2smooth}, there exist constants $C_1,C_2>0$ such that for any $k\geq 0$, $\|\hat{g}_k\|_2\leq C_1$ almost surely, $\|\nabla_{\theta}f^{\text{\rm MF-OMO}}(\theta_k)\|_2\leq C_1$, $\mathbb{E}_k\hat{g}_k=\nabla_{\theta}f^{\text{\rm MF-OMO}}(\theta_k)$ and  $\mathbb{E}_k\|\hat{g}_k\|_2^2\leq C_2$. Here $\mathbb{E}_k$ is the conditional expectation given the $k$-th iteration $\theta_k$.
\end{proposition}
That is, the estimator \eqref{hatgk} is unbiased, almost surely bounded, with a bounded second-order moment, and the exact gradient is also uniformly bounded. The proof is a direct application of the continuity of the rewards and dynamics and the compactness of $\Theta$, and  omitted here.

\subsection{Global convergence to stationary points without definability}
In this section, we show the global convergence of the iterates to stationary points when only the smoothness assumption holds. Let us first 
recap the concept of stationary points for constrained optimization problems and their relation to optimality. 
\paragraph{Optimality conditions.}  
The KKT conditions state that a necessary condition for a point $\theta\in\Theta$ to be the optimal solution to \eqref{mfvmo-full} is  $-\nabla_{\theta}f^{\text{MF-OMO}}(\theta)\in\mathcal{N}_{\Theta}(\theta)$, where
\[
\mathcal{N}_{\Theta}(\theta)=\left\{\begin{array}{ll}
   \emptyset,  & \theta\notin\Theta, \\
    \{\nu|\nu^\top(\theta'-\theta)\leq 0,\,\forall\theta'\in\Theta\}, & \theta\in\Theta
\end{array}\right.
\] 
is the normal cone of $\Theta$ at $\theta$. We say that $\theta$ is a stationary point of \eqref{mfvmo-full} if $-\nabla_{\theta}f^{\text{MF-OMO}}(\theta)\in\mathcal{N}_{\Theta}(\theta)$. 

The typical goal of nonconvex optimization is to find a sequence of points that converge to a stationary point. Since the normal cone is set-valued, it is hard to directly evaluate the closeness to stationary points with the above definition. A common surrogate metric is the norm of the projected gradient mapping $G_{\eta}:\Theta\rightarrow\Theta$,  defined as  \[
G_\eta(\theta):=\dfrac{1}{\eta}\left(\theta-\textbf{Proj}_{\Theta}(\theta-\eta\nabla_{\theta}f^{\text{MF-OMO}}(\theta))\right).
\]
Here $G_{\eta}(\theta)=0$ if and only if $\theta$ is a stationary point of \eqref{mfvmo-full}. Moreover, if $\|G_{\eta}(\theta)\|_2\leq \epsilon$, then according to \cite[Lemma 3]{ghadimi2016accelerated}, 
\[
-\nabla f^{\text{MF-OMO}}(\theta)\in \mathcal{N}_{\Theta}(\theta)+\epsilon(\eta M+1)\mathbb{B}_2,
\]
where $\mathbb{B}_2$ is the unit $\ell_2$ ball. 

When $\Theta$ is the full Euclidean space, the projected gradient mapping is simply the gradient of $f^{\text{MF-OMO}}$. We refer interested readers to \cite[Chapter 3]{beck2017first} for a more detailed discussion on optimality conditions.

\paragraph{Convergence to stationary points.} 

Under Assumption \ref{C2smooth}, the following convergence result for PGD is standard given Proposition \ref{strong_smooth_fact} and 
the existence of optimal solution(s) for \eqref{mfvmo-full}. 
\begin{proposition}{\cite[Theorem 10.15]{beck2017first}}\label{stationary_convergence_pgd}
Under Assumption \ref{C2smooth}, let $\{\theta_k\}_{k\geq 0}$ be the sequence generated by PGD \eqref{pgd_mfvmo} with $\eta_k=\eta\in(0,2/M)$. Then 
\begin{itemize}
    \item the objective value sequence $\{f^{\text{\rm MF-OMO}}(\theta_k)\}_{k\geq 0}$ is decreasing, and the decreasing is strict until a stationary point is found (\ie, $\nabla_{\theta}f^{\text{\rm MF-OMO}}(\theta_k)=0$) and the iteration is terminated;
    \item $\sum_{k=0}^K\|G_{\eta}(\theta_k)\|_2^2\leq \frac{f^{\text{\rm MF-OMO}}(\theta_0)}{\eta-M\eta^2/2}$, 
and in particular, 
\[
\min_{k=0,\dots,K}\|G_\eta(\theta_k)\|_2\leq \sqrt{\dfrac{f^{\text{\rm MF-OMO}}(\theta_0)}{(\eta-M\eta^2/2)(K+1)}};
\]
\item all limit points of $\{\theta_k\}_{k\geq 0}$ are stationary points of \eqref{mfvmo-full}.
\end{itemize} 
\end{proposition}

Similarly, the following proposition shows that the iterates of SPGD converge to stationary points with high probability. 
The uniform boundedness in Proposition \ref{hatgk_properties} allows us to adapt the proof in \cite[Theorem 9]{zhang2021sample} for policy gradient methods of MDPs to our setting. 
The proof is omitted due to similarity. 


\begin{proposition}\label{spgd_stationary_convergence}
Under Assumptions \ref{C2smooth} and \ref{definability}, let $\{\theta_k\}_{k\geq 0}$ be the sequence generated by SPGD \eqref{spgd_mfvmo} with $\eta_k=\frac{1}{\sqrt{k+3}\log_2(k+3)}$. Then for any $K\geq 0$, with probability at least $1-\delta$, we have 
\[
\min_{k=0,\dots,K}\|G_{\eta}(\theta_k)\|_2^2\leq \dfrac{4\log_2(K+3)}{\sqrt{K+1}}\left(C_2M+f^{\text{\rm MF-OMO}}(\theta_0)+C_1^2\sqrt{(16+M)\log\left(\frac{2}{\delta}\right)}\right).
\]
\end{proposition}

\subsection{Reparametrization, acceleration and variance reduction}\label{popular_tricks} 
\eqref{mfvmo-full} may be solved more efficiently. 
First, 
one can also reparametrize the variables to completely get rid of the simple constraints in $\Theta$. In particular, one can reparametrize $L$ and $z$ by 
\[
L_{s,a,t}=\frac{\exp(u_{s,a,t})}{\color{black}\sum_{s'\in\mathcal{S},a'\in\mathcal{A}}\exp(u_{s',a',t})},\quad z_{s,a,t}=SA(T^2+T+2)r_{\max}\frac{\exp(v_{s,a,t})}{\color{black}\sum_{s'\in\mathcal{S},a'\in\mathcal{A}}\exp(v_{s',a',t})+w_0}
\]
for some $u=[u_{s,a,t}]_{s\in\mathcal{S},a\in\mathcal{A},t\in\mathcal{T}}\in\mathbb{R}^{SA(T+1)}$ and {\color{black}$v=[v_{s,a,t}]_{s\in\mathcal{S},a\in\mathcal{A},t\in\mathcal{T}}\in\mathbb{R}^{SA(T+1)}$, and $w_0\in\mathbb{R}$}. Similarly, $y$ can be reparametrized using trigonometric functions such as {\color{black}$y_{s,t}=\frac{S(T+1)(T+2)r_{\max}}{2\sqrt{S(T+1)}}\sin(w_{s,t})$} for some $w\in\mathbb{R}^{S(T+1)}$.
With the aforementioned reparametrization, \eqref{mfvmo-full} becomes a smooth unconstrained optimization problem, which can be solved by additional optimization algorithms and solvers not involving projections. 
As mentioned in Remark \ref{remark_mfvmo_form}, 
one can also change the norms of $y$ and $z$ and their bands 
and apply other reparametrizations. 

Lastly, standard techniques such as momentum methods \cite{polyak1987introduction}, Nesterov acceleration \cite{nesterov2018lectures}, quasi-Newton methods \cite{wright1999numerical}, Anderson acceleration \cite{anderson1965iterative}, and Newton methods \cite{luenberger1984linear} can readily be applied to accelerate the convergence of PGD and SPGD; variance reduction approaches  such as SAG \cite{le2013stochastic}, SVRG \cite{johnson2013accelerating}, SAGA \cite{defazio2014saga}, SEGA \cite{hanzely2018sega} can also be adopted to further accelerate and stabilize the convergence of SPGD. In particular, safeguarded Anderson acceleration methods \cite{henderson2019damped,zhang2020globally,sopasakis2019superscs} might be adopted to maintain the desired monotonicity property of vanilla PGD, which is central for the local convergence to global Nash equilibrium solutions (\cf Theorem \ref{local_conv_global}).

\section{Definability and local convergence to global Nash equilibrium solutions}

\subsection{Definability}\label{definability_appendix}
In this section, we restate the formal definitions of definable sets and functions as in \cite[Section 4.3]{attouch2010proximal}, and summarize and derive necessary properties 
to prove Theorem \ref{local_conv_global}. See \cite{van1996geometric,wilkie1996model,van1998tame,marker2000tame,coste2000introduction-ominimal,kaiser2020differentiability} for more complete expositions 
of the topic. 

We first need the following concept of o-minimal structure over $\mathbb{R}$.
\begin{definition}\label{o-minimal-structure}
Suppose that $\mathcal{O}:=\{\mathcal{O}_n\}_{n=0}^{\infty}$ is a sequence of subsets collections $\mathcal{O}_n\subseteq 2^{\mathbb{R}^n}$. We say that $\mathcal{O}$ is an o-minimal structure over $\mathbb{R}$ if it satisfies the following conditions:
\begin{itemize}
    \item $\forall n\geq 0$, $\emptyset\in\mathcal{O}_n$. And $\forall A,B\in\mathcal{O}_n$,  $A\cap B, A\cup B, \mathbb{R}^n\backslash A \in\mathcal{O}_n$.
    \item $\forall A\in\mathcal{O}_n$, $A\times\mathbb{R}\in\mathcal{O}_{n+1}$ and $\mathbb{R}\times A\in\mathcal{O}_{n+1}$.
    \item $\forall A\in\mathcal{O}_{n+1}$, its canonical projection $\{(x_1,\dots,x_n)\in\mathbb{R}^n|(x_1,\dots,x_n,x_{n+1})\in A\}\in\mathcal{O}_n$. 
    \item $\forall i\neq j$ and $i,j\in\{1,\dots,n\}$, the diagonals $\{(x_1,\dots,x_n)|x_i=x_j\}\in\mathcal{O}_n$. 
    \item The set $\{(x_1,x_2)\in\mathbb{R}^2|x_1<x_2\}\in\mathcal{O}_2$.
    \item The elements of $\mathcal{O}_1$ are exactly finite unions of intervals (including both open and closed, and in particular single points). 
\end{itemize}
\end{definition}

We then have the following definition of definability. 
\begin{definition}\label{def_definable}
Let $\mathcal{O}=\{\mathcal{O}_n\}_{n=0}^{\infty}$ be an o-minimal structure over $\mathbb{R}$. We say that a set $A\in\mathbb{R}^n$ is definable if  $A\in\mathcal{O}_n$. A 
function/mapping $f:A\subseteq\mathbb{R}^n\rightarrow\mathbb{R}^m$ 
is said to be definable if its graph $\{(x,y)\in\mathbb{R}^n\times\mathbb{R}^m|y=f(x)\}\in\mathcal{O}_{n+m}$.\footnote{Note that a real extended value function $f:\mathbb{R}^n\rightarrow\mathbb{R}\cup\{+\infty\}$ can be equivalently seen as a standard function $f:\textbf{dom}(f)\rightarrow\mathbb{R}$. Here $\textbf{dom}(f)$ denotes the domain of $f$.}  A set-valued mapping $f:A\subseteq \mathbb{R}^n\rightrightarrows \mathbb{R}^m$ is said to be definable if its graph $\{(x,y)\in A\times\mathbb{R}^m|y\in f(x)\}\in\mathcal{O}_{n+m}$.
\end{definition}
 Below we provide three examples of most widely used o-minimal structures, namely the real semialgebraic structure, globally subanalytic structure and the log-exp structure used in our main text.
\begin{example}[Real semialgebraic structure]
In this structure, each $\mathcal{O}_n$ consists of real semialgebraic sets in $\mathbb{R}^n$. A set in $\mathbb{R}^n$ is said to be real semialgebraic (or semialgebraic for short) if it can be written as a finite union of sets of the form $\{x\in\mathbb{R}^n|p_i(x)=0,\,q_i(x)<0,\,i=1,\dots,k\}$, where $p_i,q_i$ are real polynomial functions and $k\geq 1$. A function/mapping $f:A\subseteq\mathbb{R}^n\rightarrow\mathbb{R}^m$ is called (real) semialgebraic if its graph is semialgebraic. 
\end{example}
The following result is 
useful for proving 
Theorem \ref{local_conv_global} (see \cite{coste2000introduction} for more properties of real semialgberaic sets). 
\begin{lemma}\label{closure_semialgebraic}
Any set of the form $\{x\in\mathbb{R}^n|p_i(x)=0,\,q_j(x)\leq 0,\,i=1,\dots,k,\,j=1,\dots,l\}$, where $p_i,q_j$ are real polynomial functions and $k,l\geq 1$ is real semialgebraic. 
\end{lemma}
\begin{proof}
First, notice that $\{x\in\mathbb{R}^n|q(x)<0\}=\{x\in\mathbb{R}^n|0=0, q(x)<0\}$ and $\{x\in\mathbb{R}^n|p(x)=0\}=\{x\in\mathbb{R}^n|p(x)=0,-1<0\}$ are both real semialgebraic by definition. Hence $\{x\in\mathbb{R}^n|p(x)\leq 0\}=\{x\in\mathbb{R}^n|p(x)<0\}\cup \{x\in\mathbb{R}^n|p(x)=0\}$ is real semialgebraic.  
Finally, by the intersection closure property of o-minimal structures, 
\[
\begin{split}
&\{x\in\mathbb{R}^n|p_i(x)=0,q_j(x)\leq 0,i=1,\dots,k,j=1,\dots,l\}\\
&=\left(\bigcap_{i=1}^k\{x\in\mathbb{R}^n|p_i(x)=0\}\right)\bigcap\left(\bigcap_{j=1}^l\{x\in\mathbb{R}^n|q_j(x)\leq 0\}\right)
\end{split}
\]
 is also real semialgebraic. 
\end{proof}

\begin{example}[Globally subanalytic structure] This structure is the smallest o-minimal structure such that each $\mathcal{O}_{n+1}$ ($n\geq 1$) includes but is not limited to all sets of the form $\{(x,t)\in [-1,1]^n\times\mathbb{R}| f(x)=t\}$, where $f:[-1,1]^n\rightarrow\mathbb{R}$ is an analytic function that can be extended analytically to a neighborhood containing $[-1,1]^n$. In particular, all semialgebraic sets and all sets of the form $\{(x,t)\in [a,b]^n\times\mathbb{R}| f(x)=t\}$ with $f:\mathbb{R}^n\rightarrow\mathbb{R}$ being analytic and $a\leq b\in\mathbb{R}$ belong to this structure. 
\end{example}

\begin{example}[Log-exp structure]\label{log-exp-structure}
This structure is the smallest o-minimal structure containing the globally analytic structure and the graph of the exponential function $\exp:\mathbb{R}\rightarrow\mathbb{R}$.
\end{example}
As in the main text, we say that a function (or a set) is definable if it is definable on the log-exp structure. The following proposition formally summarizes the major facts about definable functions and sets. See \cite{attouch2010proximal} and \cite{van1996geometric} for more details.
\begin{proposition}\label{facts_of_definable}
The following facts hold for definable functions/mappings and sets.
\begin{enumerate}
    \item Let the real extended value functions $f,g:\mathbb{R}^n\rightarrow\mathbb{R}\cup\{+\infty\}$ be definable, and let $G\in\mathbb{R}^{n\times m}$ be a constant matrix and $h\in\mathbb{R}^n$ be a constant vector. Then $f(x)+g(x)$, $f(x)-g(x)$, $f(x)g(x)$, $f(Gz+h)$, $\max\{f(x),g(x)\}$, $\min\{f(x),g(x)\}$ and $f^{-1}(y)$ are definable. The same results hold for vector-valued mappings $f,g:A\subseteq\mathbb{R}^n\rightarrow\mathbb{R}^m$.
    \item Let the functions/mappings $f:A\subseteq\mathbb{R}^n\rightarrow\mathbb{R}^m$ and $g:B\subseteq\mathbb{R}^m\rightarrow\mathbb{R}^l$, and let $f(A)\subseteq B$. Then the composition $h(x)=g(f(x)):A\rightarrow\mathbb{R}^l$ is also definable.
    \item Let $f:A\subseteq\mathbb{R}^n\rightarrow\mathbb{R}$ be definable and let $B\subseteq A$ be a compact set on which $f(x)\neq 0$ for all $x\in B$. Then $g(x)=1/f(x):B\rightarrow \mathbb{R}$ is also definable. 
    \item Let $f(x,y):\mathbb{R}^n\times\mathbb{R}^m\rightarrow\mathbb{R}$ and $C\subseteq\mathbb{R}^m$ be definable. Then $g(x)=\sup_{y\in C}f(x,y):\mathbb{R}^n\rightarrow\mathbb{R}$ and $h(x)=\inf_{y\in C}f(x,y):\mathbb{R}^n\rightarrow\mathbb{R}$ are both definable.  
    \item All semialgebraic functions,  all analytic functions restricted to definable compact sets, as well as the exponential function $\exp$ and logarithm function $\log$ are definable. 
    \item Let the set $\Theta$ be definable. Then the indicator function $\mathbf{1}_{\Theta}(\theta)$ (that takes the value $1$ when $\theta\in\Theta$ and $0$ otherwise) and the characteristic function\footnote{Here $\mathcal{I}_{\mathcal{X}}(x)$ denotes the (real extended valued) characteristic function (sometimes also referred to as the indicator function in some literature) that takes value $0$ when $x\in\mathcal{X}$ and $+\infty$ otherwise.} $\mathcal{I}_{\Theta}(\theta)$ are also definable. 
    \item Let the sets $A,B\subseteq\mathbb{R}^n$ be definable. Then the Cartesian product $A\times B$ is definable. 
\end{enumerate}
\end{proposition}

\subsection{Local convergence to global Nash equilibrium solutions}\label{local_conv_global_proof}
\begin{proof}[Proof of Theorem \ref{local_conv_global}]
The proof consists of four steps. We begin by verifying the definability around the initialization. We  
then 
specify $\epsilon_0$ and show that the iterates are uniformly bounded in a similar neighborhood of the initialization. Finally, we obtain convergence to Nash equilibrium solutions. 
\paragraph{Step 1: Verification of definability around initialization.}
We first verify that \[
h^{\text{MF-OMO}}(\theta):=f^{\text{MF-OMO}}(\theta)+\mathcal{I}_{\Theta}(\theta)
\]
is definable. Note that $\theta$ is an optimal solution to \eqref{mfvmo-full} if and only if $\theta$ minimizes $h^{\text{MF-OMO}}(\theta)$ over the entire Euclidean space of $\theta$. To see this, note that $f^{\text{MF-OMO}}(\theta)$ is obtained from a finite number of summations, subtractions and products of definable functions (\ie, $P_t(s,a,L_t)$, $r_t(s,a,L_t)$, constant functions and identity functions), and is hence definable by Proposition \ref{facts_of_definable}. In addition, by rewriting $\|y\|_2\leq S(T+1)(T+2)r_{\max}/2$ as $\sum_{i=1}^{S(T+1)}y_i^2\leq S^2(T+1)^2(T+2)^2r_{\max}^2/4$ and applying Lemma \ref{closure_semialgebraic}, $\Theta$ is obviously semialgebraic and hence definable, and hence the characteristic function $\mathcal{I}_{\Theta}(\theta)$ is also definable by Proposition \ref{facts_of_definable}. As a result, the sum of $f^{\text{MF-OMO}}(\theta)$ and   $\mathcal{I}_{\Theta}(\theta)$ is also definable. 

Let $(\pi^\star,L^\star)$ be some Nash equilibrium solution. Define $y^\star$ and $z^\star$ by \[
y^\star:=[V_1^\star(L^\star), V_2^\star(L^\star), \dots,V_T^\star(L^\star), -V_0^\star(L^\star)]\in\mathbb{R}^{S(T+1)}
\]
and $z^\star=[z_{\star,0},z_{\star,1},\dots,z_{\star,T}]\in\mathbb{R}^{SA(T+1)}$ with $z_{\star,t}=[z_{\star,t}^1,\dots,z_{\star,t}^A]$ ($t=0,\dots,T$), where
\[
z_{\star,t}^a=V_t^\star(L^\star)-r_t(\cdot,a,L_t^\star)-P_t^a(L_t^\star)V_{t+1}^\star(L^\star)\in\mathbb{R}^S, \quad t=0,\dots,T-1,\,a\in\mathcal{A}
\]
and $z_{\star,T}^a=V_T^\star(L^\star)-r_T(\cdot,a,L_T^\star)\in\mathbb{R}^S$ ($a\in\mathcal{A}$). Then by Theorem \ref{thm:nash-opt}, Proposition \ref{interpretation_y_z} and Corollary \ref{y_z_bound}, $\theta^\star$ is feasible (\ie, $\theta^\star\in\Theta$) and $A_{L^\star}L^\star=b$, $A_{L^\star}^\top y^\star+z^\star=c_{L^\star}$ and $(z^\star)^\top L^\star=0$. Hence $f^{\text{MF-OMO}}(\theta^\star)=0$ and thus $\theta^\star$ is optimal (\ie,  $\theta^\star\in\Theta^\star$). 


For this $\theta^\star\in\Theta^\star$, we have $h^{\text{MF-OMO}}(\theta^\star)=f^{\text{MF-OMO}}(\theta^\star)=0$. By \cite[Theorem 14]{attouch2010proximal} and the continuity of $f^{\text{MF-OMO}}(\theta)$, there exist $\epsilon,\delta>0$ and a continuous, concave, strictly increasing and definable function $\phi:[0,\delta)\rightarrow\mathbb{R}_+$, with $\phi(0)=0$ and $C^1$ in $(0,\delta)$, such that for any $\theta\in\Theta$ with $\|\theta-\theta^\star\|_2\leq \epsilon$ and $0<h^{\text{MF-OMO}}(\theta)<\delta$, the following Kurdyka-Lojasiewicz (KL) condition holds: 
\begin{equation}\label{kl_condition}
\phi'(h^{\text{MF-OMO}}(\theta))\textbf{dist}(0,\nabla_{\theta} f^{\text{MF-OMO}}(\theta)+\mathcal{N}_{\Theta}(\theta))\geq 1.
\end{equation}
Here we use the fact that the limiting sub-differential $\partial h^{\text{MF-OMO}}(\theta)=\nabla_{\theta} f^{\text{MF-OMO}}(\theta)+\mathcal{N}_{\Theta}(\theta)$ \cite[Proposition 3]{attouch2010proximal}.

\paragraph{Step 2: Specify $\epsilon_0$.} Next, we specify the choice of $\epsilon_0$. By the continuity of $f^{\text{MF-OMO}}(\theta)$, there exists $\epsilon'\in(0,\epsilon/3]$ such that for any $\theta\in\Theta$ with $\|\theta-\theta^\star\|_2\leq \epsilon'$, 
\[
h^{\text{MF-OMO}}(\theta)=f^{\text{MF-OMO}}(\theta)\leq \min\left\{\phi^{-1}\left(\frac{\epsilon}{3(M+1/\eta)}\right),\frac{(2-M\eta)\epsilon^2}{18\eta},\delta/2\right\}<\delta.
\]
Let $\epsilon_0:=\min\{\epsilon,\epsilon'\}/C$, where 
\[
C= SA(T+1)(C_P(T+1)^2r_{\max}+C_r(2T+1))+S(T+1)\left(\dfrac{C_PT(T+1)r_{\max}}{2}+C_rT\right)+1.
\]
We now show that for any $L^0\in\Delta(\mathcal{S}\times\mathcal{A})$ with $\|L^0-L^\star\|_1\leq \epsilon_0$,  $\|\theta_0-\theta^\star\|_2\leq C\epsilon_0=\min\{\epsilon,\epsilon'\}$, and hence $h^{\text{MF-OMO}}(\theta_0)\leq
\min\left\{\phi^{-1}(\frac{\epsilon}{3(M+1/\eta)}),\frac{(2-M\eta)\epsilon^2}{18\eta},\delta/2\right\}
<\delta$.

By Assumption \ref{C2smooth} and the compactness of $\Theta$, the Lipschitz continuity assumptions in Theorem \ref{VMO-vs-MFNE} hold for some constants $C_P,C_r>0$. Since  $\|L^0-L^\star\|_1\leq \epsilon_0$, by the proof of Proposition \ref{perturbation} and the construction of $y^0$ and $z^0$,  $\theta_0=(y^0,z^0,L^0)\in\Theta$ is feasible by the proof of Theorem \ref{MFNE-vs-VMO}, and 
\[
\begin{split}
\|y^0-y^\star\|_{\infty}&=\max_{t\in\mathcal{T}}\|V_t^\star(L^0)-V_t^\star(L^\star)\|_{\infty}\\
&\leq \dfrac{T(T+1)r_{\max}}{2}\max_{t=0,\dots,T-1}\|P_t^{L^0}-P_t^{L^\star}\|_{\infty,1}+\sum_{t\in\mathcal{T}}\|r_t^{L^0}-r_t^{L^\star}\|_{\infty}\\
&\leq \left(\dfrac{C_PT(T+1)r_{\max}}{2}+C_rT\right)\epsilon_0.
\end{split}
\]
Similarly, 
\[
\begin{split}
\|z^0-z^\star\|_{\infty}&\leq\max_{t\in\mathcal{T}}\|V_t^\star(L^0)-V_t^\star(L^\star)\|_{\infty}+\max_{a\in\mathcal{A},t\in\mathcal{T}}\|r_t(\cdot,a,L_t^0)-r_t(\cdot,a,L_t^\star)\|_{\infty}\\
&\qquad+\max_{a\in\mathcal{A},t\in\mathcal{T}}\|P_t^a(L_t^0)V_{t+1}^\star(L^0)-P_t^a(L_t^\star)V_{t+1}(L^\star)\|_{\infty}\\
&\leq \left(\dfrac{C_PT(T+1)r_{\max}}{2}+C_rT\right)\epsilon_0+C_r\epsilon_0\\
&\qquad+\max_{a\in\mathcal{A},t\in\mathcal{T}}\|P_t^a(L_t^0)\|_\infty\|V_{t+1}^\star(L^0)-V_{t+1}^\star(L^\star)\|_\infty \\
&\qquad+\max_{a\in\mathcal{A},t\in\mathcal{T}}\|P_t^a(L_t^0)-P_t^a(L_t^\star)\|_\infty\|V_{t+1}^\star(L^\star)\|_\infty \\
&\leq \left(\dfrac{C_PT(T+1)r_{\max}}{2}+C_r(T+1)\right)\epsilon_0+\left(\dfrac{C_PT(T+1)r_{\max}}{2}+C_rT\right)\epsilon_0\\
&\qquad + r_{\max}(T+1)C_P\epsilon_0\\
&=(C_P(T+1)^2r_{\max}+C_r(2T+1))\epsilon_0,
\end{split}
\]
where $V_{T+1}(L)=0$ for any $L\in\Delta(\mathcal{S}\times\mathcal{A})$. Here we use the fact that $\|P_t^a(L_t^0)\|_\infty=1$,  
$\max_{a\in\mathcal{A},t\in\mathcal{T}}\|P_t^a(L_t^0)-P_t^a(L_t^\star)\|_\infty=\|P^{L^0}-P^{L^\star}\|_{\infty,1}$, and $\max_{t\in\mathcal{T}}|V_t^\star(L^\star)(s)|\leq r_{\max} (T+1)$. 
Also note that here $\|\cdot\|_1$ and $\|\cdot\|_\infty$ are both vector norms (except for $P_t^a(L_t)$, for which $\|\cdot\|_\infty$ is the matrix $\ell_\infty$-norm). 

Hence we have shown that 
\[
\|\theta_0-\theta^\star\|_2\leq 
\|\theta_0-\theta^\star\|_1=\|z^0-z^\star\|_1+\|y^0-y^\star\|_1+\|L^0-L^\star\|_1\leq C\epsilon_0=\min\{\epsilon,\epsilon'\},
\]
where $C$ is a problem dependent constant defined as above. Thus 
$h^{\text{MF-OMO}}(\theta_0)<\delta$. 



Now by Proposition \ref{stationary_convergence_pgd} and the feasibility of $\{\theta_k\}_{k\geq 0}$, $h^{\text{MF-OMO}}(\theta_k)=f^{\text{MF-OMO}}(\theta_k)$ is non-increasing, and hence 
\[
h^{\text{MF-OMO}}(\theta_k)\leq h^{\text{MF-OMO}}(\theta_0)\leq\min\left\{\phi^{-1}\left(\frac{\epsilon}{3(M+1/\eta)}\right),\frac{(2-M\eta)\epsilon^2}{18\eta},\delta/2\right\}
\leq \delta/2<\delta.
\]

\paragraph{Step 3: Show that $\|\theta_k-\theta^\star\|_2\leq \epsilon$, $\|\theta_k-\theta^\star\|_2= O(g(\epsilon_0))$ for some $g$ with $\lim_{\tilde{\epsilon}\rightarrow0}g(\tilde{\epsilon})=0$.} 
Define $f_k=f^{\text{MF-OMO}}(\theta_k)$. 
By \cite[Lemma 10.4]{beck2017first}, 
\begin{equation}\label{descent_pgd}
f_k-f_{k+1}=f^{\text{MF-OMO}}(\theta_k)-f^{\text{MF-OMO}}(\theta_{k+1})\geq \left(\frac{1}{\eta}-\frac{M}{2}\right)\|\theta_{k+1}-\theta_k\|_2^2.
\end{equation}
Then since $\phi$ (in the KL condition \eqref{kl_condition} above) is concave, strictly increasing and $C^1$ in $(0,\delta)$, we have $\phi'(f^{\text{MF-OMO}}(\theta_k))>0$ as long as $f^{\text{MF-OMO}}(\theta_k)>0$ (since we already have $f^{\text{MF-OMO}}(\theta_k)<\delta$ for all $k\geq 0$), and 
\begin{equation}\label{kl_concave_ineq}
\begin{split}
\phi(f_k)-\phi(f_{k+1}))&\geq \phi'(f_k)(f_k-f_{k+1})\geq \phi'(f_k)\left(\frac{1}{\eta}-\frac{M}{2}\right)\|\theta_{k+1}-\theta_k\|_2^2.
\end{split}
\end{equation}
Here the first inequality is by 
the concavity of $\phi$.

Now we prove that 
\begin{equation}\label{induction_claim_step3}
    \|\theta_i-\theta^\star\|_2\leq \min\{g(\epsilon_0),\epsilon\} 
    \end{equation}
    for all $i\geq 0$ by induction, where 
\begin{equation}\label{g_def}
\begin{split}
g(\epsilon_0)&=(M+1/\eta)\phi(C_f\epsilon_0)+\sqrt{\frac{2\eta}{2-M\eta}C_f\epsilon_0}+C\epsilon_0\\
&\geq(M+1/\eta)\phi(f_0)+\sqrt{\frac{2\eta}{2-M\eta}f_0}+C\epsilon_0,
\end{split}
\end{equation}
where $C_f>0$ is the Lipschitz constant of $f^{\text{MF-OMO}}(\theta)$ in the sense that $|f^{\text{MF-OMO}}(\theta)-f^{\text{MF-OMO}}(\theta')|\leq C_f\|\theta-\theta'\|_2$ for any $\theta,\theta'\in\Theta$. The existence of $C_f$ is guaranteed by Assumption \ref{C2smooth} and the compactness of $\Theta$. Note that since $\lim_{\epsilon_0\rightarrow0}f_0=0$ due to the continuity of $f^{\text{MF-OMO}}(\cdot)$, we have the desired property $\lim_{\tilde{\epsilon}\rightarrow0}g(\tilde{\epsilon})=0$. 


The base case for $i=0$ is trivial. 
Without loss of generality, we now assume that $f^{\text{MF-OMO}}(\theta_0)>0$, since otherwise the algorithm finds the exact Nash equilibrium solution at the initialization and the iterates terminate/stay unchanged thereafter, and hence all the claims in Theorem \ref{local_conv_global} obviously hold. 

Then for $i=1$, by \eqref{descent_pgd}, 
\[
\begin{split}
\|\theta_1-\theta^\star\|_2&\leq \|\theta_0-\theta^\star\|_2+\|\theta_1-\theta_0\|_2\\
&\leq C\epsilon_0+\sqrt{\frac{2\eta}{2-M\eta}(f_0-f_1)}\leq C\epsilon_0+\sqrt{\frac{2\eta}{2-M\eta}f_0}\\
&\leq \left\{
\begin{array}{l}
g(\epsilon_0),\\
\dfrac{\epsilon}{3}+\dfrac{\epsilon}{3}\leq \epsilon. 
\end{array}\right.
\end{split}
\]

Suppose that the claim \eqref{induction_claim_step3} holds 
for $i=0,\dots,k$ ($k\geq 0$). Without loss of generality, we also assume that $f^{\text{MF-OMO}}(\theta_k),f^{\text{MF-OMO}}(\theta_{k+1})>0$, since otherwise the algorithm finds the exact Nash equilibrium solution at iteration $k$ or $k+1$ and the iterates terminate/stay unchanged thereafter, and hence all the claims in Theorem \ref{local_conv_global} obviously hold. 


Now for any $i\in\{1,\dots,k\}$, since \[
\theta_{i}\in\text{argmin}_{\theta}\,\mathcal{I}_{\Theta}(\theta)+\frac{1}{2}\|\theta-(\theta_{i-1}-\eta\nabla_{\theta} f^{\text{MF-OMO}}(\theta_{i-1}))\|_2^2,
\]
we have 
\[
0\in\mathcal{N}_{\Theta}(\theta_{i})+\theta_{i}-\theta_{i-1}+\eta \nabla_{\theta}f^{\text{MF-OMO}}(\theta_{i-1}).
\]
Hence  $\frac{1}{\eta}(\theta_{i-1}-\theta_i)-\nabla_{\theta}f^{\text{MF-OMO}}(\theta_{i-1})\in\mathcal{N}_{\Theta}(\theta_i)$ since $\mathcal{N}_{\Theta}(\theta_i)$ is a cone, and  by the KL condition at $\theta_i\in\Theta$ with $\|\theta_i-\theta^\star\|_2\leq\epsilon$ and $h^{\text{MF-OMO}}(\theta_i)\in(0,\delta)$, we have
\begin{equation}\label{gradient_property}
\phi'(f_i)\left\|
\frac{1}{\eta}(\theta_{i-1}-\theta_i)-\nabla_{\theta}f^{\text{MF-OMO}}(\theta_{i-1})+\nabla_{\theta}f^{\text{MF-OMO}}(\theta_i)\right\|_2\geq 1.
\end{equation}
\eqref{gradient_property}, together with \eqref{kl_concave_ineq}, immediately implies that
\[
\begin{split}
1&\leq \phi'(f_i)\left(\frac{1}{\eta}\|\theta_i-\theta_{i-1}\|_2+M\|\theta_i-\theta_{i-1}\|_2\right)\\
&\leq \left(\frac{1}{\eta}+M\right)(\phi(f_i)-\phi(f_{i+1}))\dfrac{\|\theta_i-\theta_{i-1}\|_2}{\|\theta_{i+1}-\theta_i\|_2^2},
\end{split}
\]
thus 
\[
\begin{split}
2\|\theta_{i+1}-\theta_i\|_2&\leq 2\sqrt{(M+1/\eta)(\phi(f_i)-\phi(f_{i+1}))}\sqrt{\|\theta_i-\theta_{i-1}\|_2}\\
&\leq (M+1/\eta)\phi(f_i)-\phi(f_{i+1}))+\|\theta_i-\theta_{i-1}\|_2.
\end{split}
\]

By telescoping the above inequality for $i=1,\dots,k$,
\[
\|\theta_{k+1}-\theta_k\|_2+\sum_{i=1}^k\|\theta_{i+1}-\theta_i\|_2\leq (M+1/\eta)\phi(f_1)+\|\theta_1-\theta_0\|_2,
\]
from which we deduce that
\[
\begin{split}
\|\theta_{k+1}-\theta^\star\|_2&\leq \sum_{i=1}^k\|\theta_{i+1}-\theta_{i}\|_2+\|\theta_1-\theta^\star\|_2\\
&\leq (M+1/\eta)\phi(f_0)+\|\theta_1-\theta_0\|_2+\|\theta_1-\theta^\star\|_2\\
&\leq (M+1/\eta)\phi(f_0)+\sqrt{\frac{2\eta}{2-M\eta}f_0}+C\epsilon_0.\\
&\leq\begin{cases}
g(\epsilon_0),\\
\dfrac{\epsilon}{3}+\dfrac{\epsilon}{3}+\dfrac{\epsilon}{3}=\epsilon.
\end{cases}
\end{split}
\]
Note that here we use \eqref{descent_pgd} for $k=0$ and we also use the fact that $f_1\leq f_0$ and $\phi$ is increasing. This completes the induction.

\paragraph{Final step: Convergence to Nash equilibrium solutions.}

Let $\bar{\theta}$ be an arbitrary limit point of $\{\theta_k\}_{k\geq 0}$. Then $\bar{\theta}\in\Theta$, and by Proposition \ref{stationary_convergence_pgd}, $\bar{\theta}$ is a stationary point of \eqref{mfvmo-full}, \ie, $0\in\nabla_{\theta} f^{\text{MF-OMO}}(\bar{\theta})+\mathcal{N}_{\Theta}(\bar{\theta})$. Moreover, again by Proposition \ref{stationary_convergence_pgd} and the feasibility of $\{\theta_k\}_{k\geq 0}$, $h^{\text{MF-OMO}}(\theta_k)=f^{\text{MF-OMO}}(\theta_k)$ is non-increasing, and hence by the continuity of $f^{\text{MF-OMO}}(\theta)$,  
\[
h^{\text{MF-OMO}}(\bar{\theta})=\lim_{k\rightarrow\infty}h^{\text{MF-OMO}}(\theta_k)=\lim_{k\rightarrow\infty}f^{\text{MF-OMO}}(\theta_k)=f^{\text{MF-OMO}}(\bar{\theta})\leq f^{\text{MF-OMO}}(\theta_0)<\delta.
\]


In addition, by the first three steps above, we have $\|\bar{\theta}-\theta^\star\|_2\leq \epsilon$.  
Hence 
if $h^{\text{MF-OMO}}(\bar{\theta})>0$, then invoking the KL condition at $\bar{\theta}$ yields 
\[
\phi'(h^{\text{MF-OMO}}(\bar{\theta}))\text{dist}(0,\nabla_{\theta} f^{\text{MF-OMO}}(\bar{\theta})+\mathcal{N}_{\Theta}(\bar{\theta}))\geq 1.
\]
However, we also have $\text{dist}(0,\nabla_{\theta} f^{\text{MF-OMO}}(\bar{\theta})+\mathcal{N}_{\Theta}(\bar{\theta}))=0$, which is a contradiction. Hence $\bar{\theta}\in\Theta$ satisfies $f^{\text{MF-OMO}}(\bar{\theta})=0$, which indicates that $\bar{\theta}\in\Theta^\star$. Therefore $\lim_{k\rightarrow\infty}\textbf{dist}(\theta_k,\Theta^\star)=0$, and by the equivalence result in Theorem \ref{thm:nash-opt}, any $\bar{\pi}\in\Pi(\bar{L})$ with $\bar{\theta}=(\bar{y},\bar{z},\bar{L})$ is a Nash equilibrium solution of the mean-field game (\eqref{setup:MFG} and \eqref{consistency_pop_flow}). In addition, we have $\|L^k-L^\star\|_2\leq\|\theta_k-\theta^\star\|_2\leq g(\epsilon_0)$ for all $k\geq 0$. The fact that the limit point set is nonempty comes simply from the boundedness of $\{\theta_k\}_{k\geq 0}$. Moreover, the vanishing limit of the exploitabilty of $\pi^k\in\Pi(L^k)$ is an immediate implication of Theorem \ref{VMO-vs-MFNE} and the fact that $f^{\text{MF-OMO}}(\theta_k)\rightarrow0$ as $k\rightarrow\infty$. The claim about isolated Nash equilibrium solutions is obvious: simply choose $\epsilon_0$ for $g(\epsilon_0)$ such that no other Nash equilibrium solutions exist in the $g(\epsilon_0)$-neighborhood of $L^\star$.
\end{proof}

\end{document}